\documentclass{article}

\usepackage[utf8]{inputenc}
\usepackage[T1]{fontenc}
\usepackage{amsmath}
\usepackage{amssymb}
\usepackage{amsthm}
\usepackage[english]{babel}

\allowdisplaybreaks[3]

\newtheorem{theorem}{Theorem}[section]
\newtheorem{corollary}[theorem]{Corollary}
\newtheorem{lemma}[theorem]{Lemma}
\newtheorem{proposition}[theorem]{Proposition}
\newtheorem{example}[theorem]{Example}
\theoremstyle{definition}
\newtheorem{definition}[theorem]{Definition}
\theoremstyle{remark}
\newtheorem{remark}[theorem]{Remark}
\numberwithin{equation}{section}

\newcommand{\bracket}[1]{\left[#1\right]}
\newcommand{\para}[1]{\left(#1\right)}
\newcommand{\K}{\mathbb{K}}
\newcommand{\N}{\mathbb{N}}

\title{On $(n+1)$-Hom-Lie Algebras Induced by $n$-Hom-Lie Algebras}
\author{Abdennour Kitouni \and Abdenacer Makhlouf \and Sergei Silvestrov}

\begin{document}

\date{}
\maketitle
%%%%%%%%%%%%%%%%%%%%%%%%%%%%%%%%%%%%

\begin{abstract}
The purpose of this paper is to study  the relationships between an $n$-Hom-Lie algebra and its induced $(n+1)$-Hom-Lie algebra. We provide an overview of the theory and explore  the structure properties such as ideals, center, derived series, solvability, nilpotency, central extensions, and the cohomology.
\end{abstract}
\begin{small}
\end{small}

%%%%%%%%%%%%%%%%%%%%%%%%%%%%%%%%%%%%%
\section*{Introduction}
In this paper, we investigate $(n+1)$-Hom-Lie algebras constructed from $n$-Hom-Lie algebras and generalized trace maps. They are called $(n+1)$-Hom-Lie algebras induced by $n$-Hom-Lie algebras.

Ternary Lie algebras appeared first in Nambu's generalization of Hamiltonian mechanics \cite{Nambu:GenHD} which uses a generalization of Poisson algebras with a ternary bracket. The algebraic formulation of Nambu mechanics is due to Takhtajan while the structure of $n$-Lie algebra was studied by Filippov \cite{Filippov:nLie} then completed by Kasymov in \cite{Kasymov:nLie}, where solvability and nilpotency properties were studied. %%%

The cohomology of $n$-Lie algebras, generalizing the Chevalley-Eilenberg Lie algebras cohomology, was first introduced by Takhtajan \cite{Takhtajan:cohomology} in its simplest form, later a complex adapted to the study of formal deformations was introduced by Gautheron \cite{Gautheron:Rem}, then reformulated by Daletskii and Takhtajan \cite{Dal_Takh} using the notion of base Leibniz algebra of an $n$-Lie algebra. In \cite{akms:ternary}, the structure and cohomology of $3$-Lie algebras induced by Lie algebras has been investigated. 

In \cite{almy:quantnambu}, the authors introduced a realization of the quantum Nambu bracket in terms of matrices (using the commutator and the trace of matrices). This construction was generalized in \cite{ams:ternary} to the case of any Lie algebra where the commutator is replaced by the Lie bracket, and the matrix trace is replaced by linear forms having similar properties. One obtains  ternary brackets which define  $3$-Lie algebras, called $3$-Lie algebras induced by Lie algebras. In \cite{AtMaSi:GenNambuAlg}, generalizations of $n$-ary algebras of Lie
type and associative type by twisting the identities using linear maps
have been introduced. These generalizations include $n$-ary
Hom-algebra structures generalizing the $n$-ary algebras of Lie type
including $n$-ary Nambu algebras, $n$-ary Nambu-Lie algebras and
$n$-ary Lie algebras, and $n$-ary algebras of associative type
including $n$-ary totally associative and $n$-ary partially
associative algebras. In \cite{ams:ternary}, a method was demonstrated of how to construct
ternary multiplications from the binary multiplication of a Hom-Lie algebra, a linear twisting map, and a trace function satisfying certain compatibility conditions; and it was shown that this method can be used to construct
ternary Hom-Nambu-Lie algebras from Hom-Lie algebras.
This construction was generalized to $n$-Lie algebras and $n$-Hom-Nambu-Lie algebras in \cite{ams:n}, where the authors presented a construction of $(n+1)$-Hom-Nambu-Lie algebras from $n$-Hom-Nambu-Lie algebras equipped with a generalized trace function, shown that implications of the compatibility conditions, that are necessary for this construction, can be understood in terms of the kernel of the trace function and the range of the twisting maps and investigated the possibility of defining $(n+k)$-Lie algebras from $n$-Lie algebras and a $k$-form satisfying certain conditions.

The aim of this paper is to  study  the relationships between an $n$-Hom-Lie algebra and its induced $(n+1)$-Hom-Lie algebra. We explore  the structure properties of objects such as ideals, center, derived series, solvability, nilpotency, central extensions, and the cohomology.
In Section \ref{sec:preliminaries}, we recall the basic notions for $n$-Hom-Lie algebras, $n$-Hom-Lie algebras cohomology and the construction of $(n+1)$-Lie algebras induced by $n$-Lie algebras, we also give a new construction theorem for such algebras and give some basic properties. In Section \ref{sec:solvnil} we define the notion of solvability and nilpotency for $n$-Hom-Lie algebra and discuss solvability and nilpotency of $(n+1)$-Hom-Lie algebras induced by $n$-Hom-Lie algebras. In Section \ref{sec:centralextentions}, we recall the definition and main properties of central extensions of $n$-Hom-Lie algebras, and then we study central extensions of $(n+1)$-Hom-Lie algebras induced by $n$-Hom-Lie algebras. The Section \ref{sec:cohomology} is dedicated to study the corresponding cohomology. Finally, in Section \ref{sec:examples} some examples are presented. 

%%%%%%%%%%%%%%%%%%%%%%%%%%%%%%%%%%%%%%%%%%%
%%%%%%%%%%%%%%%%%%%%%%%%%%%%%%%%%%%%%%%%%%%

\section{Preliminaries} \label{sec:preliminaries}%%%
All vector spaces are over a field $\mathbb{K}$ of characteristic $0$.
\subsection{Hom-Lie algebras and $n$-Hom-Lie algebras}
Hom-Lie algebras are a generalization of Lie algebras introduced in \cite{HLS} while studying $\sigma$-derivations. The $n$-ary case was introduced in \cite{AtMaSi:GenNambuAlg}.

\begin{definition}[\cite{HLS,ms:homstructure}]
A Hom-Lie algebra is a vector space $A$ together with a bilinear map $[\cdot,\cdot] : A\times A \to A$ and a linear map $\alpha : A \to A$ satisfying:
\begin{itemize}
\item Skew-symmetry, that is $[x,y] = - [y,x]$, for all $x,y \in A$.
\item Hom-Jacobi identity
$ [\alpha(x),[y,z]] = [[x,y],\alpha (z)] + [\alpha (y) [x,z]], \forall x,y,z \in A. $
\end{itemize}
\end{definition}

\begin{definition}[\cite{HLS,LS:quasi-hom-lie}]
Let $(A,[\cdot,\cdot],\alpha)$, $(B,\{\cdot,\cdot\},\beta)$ be Hom-Lie algebras. A Hom-Lie algebra morphism is a linear map $f : A \to B$ satisfying the conditions:
\begin{itemize}
\item $f([x,y]) = \{f(x),f(y)\}$, for all $x,y \in A$.
\item $f\circ \alpha = \beta \circ f$.
\end{itemize}
A linear map satisfying only the first condition is called a weak morphism.
\end{definition}

\begin{definition}[\cite{Hombiliform,ms:homstructure}]
\begin{itemize}
\item A Hom-Lie algebra $(A,[\cdot,\cdot],\alpha)$ is said to be multiplicative if $\alpha$ is an algebra morphism.
\item It is said to be regular if $\alpha$ is an isomorphism.
\end{itemize}
\end{definition}

\begin{definition}[\cite{AtMaSi:GenNambuAlg}]
An $n$-Hom-Lie algebra is a vector space $A$ together with a $n$-linear map $[\cdot,...,\cdot] : A^n \to A$ and $(n-1)$ linear maps $\alpha_i : A \to A, 1 \leq i \leq n-1$ satisfying:
\begin{itemize}
\item Skew-symmetry, that is $[x_{\sigma(1)},...,x_{\sigma(n)}] = sgn(\sigma) [x_1,...,x_n]$,  $\forall x_1,...,x_n \in A$.
\item Hom-Nambu identity:
\begin{align*}
& [\alpha_1(x_1),...,\alpha_{n-1}(x_{n-1}),[y_1,...,y_n]] \\ &= \sum_{i=1}^n [\alpha_1(y_1),...,\alpha_{i-1}(y_{i-1}),[x_1,...,x_{n-1},y_i], \alpha_{i}(y_{i+1}),...,\alpha_{n-1}(y_{n})],\\
&\forall x_1,...,x_{n-1},y_1,...,y_n \in A. 
\end{align*} 
\end{itemize}
\end{definition}

\begin{example}
\begin{itemize}
\item If we take $\alpha_i=Id_A$ for all $1\leq i \leq n-1$, we get an $n$-Lie algebra (\cite{Filippov:nLie}). Therefore, the class of $n$-Lie algebras is included in the class of $n$-Hom-Lie algebras.
\item For any vector space $A$, if we take $\bracket{x_1,...,x_n}_0=0$ for all $x_1,..,x_n \in A$ and any linear maps $\alpha_1,...,\alpha_{n-1}$, we have that $(A, \bracket{\cdot,...,\cdot}_0,\alpha_1,...,\alpha_{n-1})$ is an $n$-Hom-Lie algebra.
\end{itemize}
\end{example}

\begin{example}
Let $A$ be an $n$-dimensional vector space, and $(e_i)_{1\leq i \leq n}$ a basis of $A$. We define the skew-symmetric $n$-linear map $\bracket{\cdot,...,\cdot}$ on $A$ by $\bracket{e_1,...,e_n}=e_1$. For any set of  $(n-1)$ linear maps $\alpha_1,...,\alpha_{n-1}$ on $A$, $(A,\bracket{\cdot,...,\cdot}, \alpha_1,...,\alpha_{n-1})$ is an $n$-Hom-Lie algebra. Indeed, we have:
\begin{align*}
&\bracket{\alpha_1(e_1),...,\alpha_{i-1}(e_{i-1}),\alpha_i(e_{i+1}),...,\alpha_{n-1}(e_n),\bracket{e_1,...,e_n}} \\
&- \sum_{j=1}^n \bracket{\alpha_1(e_1),...,\alpha_{j-1}(e_{j-1}),\bracket{e_1,...,e_{i-1},e_{i+1},...,e_n,e_j},\alpha_j(e_{j+1}),...,\alpha_{n-1}(e_n)}\\
& = \bracket{\alpha_1(e_1),...,\alpha_{i-1}(e_{i-1}),\alpha_i(e_{i+1}),...,\alpha_{n-1}(e_n),\bracket{e_1,...,e_n}} \\
&- \sum_{\substack{j=1 \\ j \neq i}}^n \bracket{\alpha_1(e_1),...,\alpha_{j-1}(e_{j-1}),\bracket{e_1,...,e_{i-1},e_{i+1},...,e_n,e_j},\alpha_j(e_{j+1}),...,\alpha_{n-1}(e_n)}\\
& -\bracket{\alpha_1(e_1),...,\alpha_{i-1}(e_{i-1}),\bracket{e_1,...,e_{i-1},e_{i+1},...,e_n,e_i},\alpha_i(e_{i+1}),...,\alpha_{n-1}(e_n)}\\
&= \bracket{\alpha_1(e_1),...,\alpha_{i-1}(e_{i-1}),\alpha_i(e_{i+1}),...,\alpha_{n-1}(e_n),\bracket{e_1,...,e_n}} \\
&- (-1)^{n-i}\bracket{\alpha_1(e_1),...,\alpha_{i-1}(e_{i-1}),\alpha_i(e_{i+1}),...,\alpha_{n-1}(e_n),(-1)^{n-i}\bracket{e_1,...,e_n}} \\
&= 0
\end{align*}
\end{example}

\begin{definition}[\cite{AtMaSi:GenNambuAlg,DYau}]
Let $(A,[\cdot,...,\cdot],\alpha_1,...,\alpha_{n-1})$, $(B,\{\cdot,...,\cdot\},\beta_1,...,\beta_{n-1})$ be $n$-Hom-Lie algebras. An $n$-Hom-Lie algebra morphism is a linear map $f : A \to B$ satisfying the conditions:
\begin{itemize}
\item $f([x_1,...,x_n]) = \{f(x_1),...,f(x_n)\}$, for all $x_1,...,x_n \in A$.
\item $f\circ \alpha_i = \beta_i \circ f$, for all $i : 1 \leq i \leq n-1$.
\end{itemize}
A linear map satisfying only the first condition is called a weak morphism.
\end{definition}

\begin{definition}[\cite{DYau}]
We refer to an $n$-Hom-Lie algebra $(A,[\cdot,...,\cdot],\alpha_1,...,\alpha_{n-1})$ such that $\alpha_1=\alpha_2=....=\alpha_{n-1}=\alpha$ by $(A,\bracket{\cdot,...,\cdot},\alpha)$. 
\begin{itemize}
\item It is said to be multiplicative if $\alpha$ is an algebra morphism.
\item It is said to be regular if it is multiplicative and $\alpha$ is an isomorphism.
\end{itemize}
\end{definition}

\begin{definition}[\cite{n-ary hom rep}]
Let $\para{A,\bracket{\cdot,...,\cdot}, \alpha}$ be a multiplicative $n$-Hom-Lie algebra and let $L(A)=\wedge^{n-1}A$ be the $(n-1)$th exterior power of $A$. The elements of $L(A)$ are called fundamental objects.

For $X = x_1\wedge ... \wedge x_{n-1},Y=y_1\wedge ... \wedge y_{n-1} \in L(A)$, we define:
\begin{itemize}
\item The map $\bar{\alpha} : \wedge^{n-1} A \to \wedge^{n-1} A $ by
$ \bar{\alpha}(X) = \alpha(x_1)\wedge...\wedge \alpha(x_{n-1}).$
\item The action of fundamental objects on $A$ by: 
\[ \forall z \in A,\ X \cdot z = ad_X(z) = \bracket{x_1,...,x_{n-1},z}. \]
\item The multiplication (composition) of two fundamental objects by: 
\[ [X,Y]_\alpha = X \cdot_\alpha Y = \sum_{i=1}^{n-1}  \alpha(y_1)\wedge ...\wedge X \cdot y_i \wedge ...\wedge \alpha(y_{n-1}). \]
\end{itemize}
We extend the preceding definitions to the entire space $L(A)$ by linearity.
\end{definition}

\begin{proposition}[\cite{n-ary hom rep}]
The space $L(A)$ equipped with the bracket $[\cdot,\cdot]_\alpha$ defined above is a Hom-Leibniz algebra. That is the bracket $[\cdot,\cdot]_\alpha$ satisfies the following identity:
\[ [\bar{\alpha}(X), [Y , Z]_\alpha]_\alpha = [[X , Y]_\alpha , \bar{\alpha}(Z)]_\alpha + [\bar{\alpha}(Y) , [X , Z]_\alpha]_\alpha. \]

\end{proposition}

%%%%%%%%%%%%%%!!!
The following proposition gives a way to construct an $n$-Hom-Lie algebra from an $n$-Lie algebra and an algebra morphism, it was first introduced in the case of Lie algebras, then generalized to the $n$-ary case in \cite{AtMaSi:GenNambuAlg}. A more general version of this theorem is given in \cite{DYau}, states that the category of $n$-Hom-Lie algebras is closed under twisting by weak morphisms:
\begin{proposition}[\cite{AtMaSi:GenNambuAlg, DYau}] \label{twist}%%%%%%
Let $\para{A,\bracket{\cdot,...,\cdot},\alpha}$ be an $n$-Hom-Lie algebra, $\beta : A \to A$ an algebra weak morphism, we define $\bracket{\cdot,...,\cdot}_\beta$ by:
\[\bracket{x_1,...,x_n}_\beta = \beta\para{\bracket{x_1,...,x_n}}.\]
We have that $\para{A,\bracket{\cdot,...,\cdot}_\beta,\beta \circ \alpha}$ is an $n$-Hom-Lie algebra. Moreover if $\para{A,\bracket{\cdot,...,\cdot},\alpha}$ is multiplicative and $\beta \circ \alpha = \alpha \circ \beta$ then $\para{A,\bracket{\cdot,...,\cdot}_\beta,\beta \circ \alpha}$ is multiplicative.
\end{proposition}

And we have the particular case for $n$-Lie algebras:

\begin{corollary}
Let $\para{A,\bracket{\cdot,...,\cdot}}$ be an $n$-Lie algebra, $\alpha : A \to A$ an algebra morphism, we define $\bracket{\cdot,...,\cdot}_\alpha$ by
$ \bracket{x_1,...,x_n}_\alpha = \alpha\para{\bracket{x_1,...,x_n}}.$\\
We have that $\para{A,\bracket{\cdot,...,\cdot}_\alpha,\alpha}$ is a multiplicative $n$-Hom-Lie algebra.
\end{corollary}

\begin{definition}[\cite{Hombiliform,ms:homstructure,DYau}]%!!!!!!!!!!!!!!
Let $(A,[\cdot,...,\cdot],\alpha_1,...,\alpha_{n-1})$ be an $n$-Hom-Lie algebra. An $n$-Hom-Lie subalgebra is a subspace $B$ of $A$ such that:
\begin{itemize}
\item $\alpha_i(B) \subseteq B$, for all $i : 1\leq i \leq n-1$.
\item For all $x_1,...,x_n \in B$ we have $[x_1,...,x_n]\in B$.
\end{itemize}
\end{definition}

\begin{definition}[\cite{Hombiliform,ms:homstructure,DYau}]%%%!!!!!!!!
Let $(A,[\cdot,...,\cdot],\alpha_1,...,\alpha_{n-1})$ be an $n$-Hom-Lie algebra. An ideal of $(A,[\cdot,...,\cdot],\alpha_1,...,\alpha_{n-1})$ is a subspace $I$ of $A$ such that:
\begin{itemize}
\item $\alpha_i(I) \subseteq I$, for all $i : 1\leq i \leq n-1$.
\item For all $x_1,...,x_{n-1} \in A$, and $y \in I$ we have $[x_1,...,x_{n-1},y]\in I$.
\end{itemize}
\end{definition}

%%%%%%%%%%%%%%%%%%%%%%%%%%%%%%%%%%%%%

\subsection{Cohomology of $n$-Hom-Lie algebras}%%Introduction modified !!!
Cohomology complexes for Hom-Lie algebras were introduced, in the multiplicative case, in \cite{sheng:hom rep} together with generalized derivations and central extensions. Deformations and relevant cohomology complex were introduced independently in \cite{hom deformation}. These notions were then generalized to the $n$-ary case in \cite{n-ary hom rep}.

In the following, for a multiplicative $n$-Hom-Lie algebra $(A,\bracket{\cdot,...,\cdot},\alpha)$ and for any $k\in \N$, we define $\alpha^k$ by
$ \alpha^0=Id ; \alpha^k=\alpha \circ \alpha \circ ... \circ \alpha (k \text{ times})\text{ for }k\neq 0. $
%If $(A,\bracket{\cdot,...,\cdot},\alpha)$ is regular, we define
\begin{definition}[\cite{sheng:hom rep,n-ary hom rep}]
Let $\para{A, \bracket{\cdot,...,\cdot},\alpha}$ be a multiplicative $n$-Hom-Lie algebra. A linear map $f : A \to A$ is said to be an $\alpha^k$-derivation if it satisfies the following conditions:
\[ f\circ \alpha = \alpha \circ f \] and \[f\para{\bracket{x_1,...,x_n}}=\sum_{i=1}^n \bracket{\alpha^k(x_1),...,f(x_i),...,\alpha^k(x_n)}, \forall x_1,...,x_n \in A. \]
\end{definition}

\begin{definition}[\cite{sheng:hom rep,n-ary hom rep}]
Let $\para{A, \bracket{\cdot,...,\cdot},\alpha}$ be a multiplicative $n$-Hom-Lie algebra. For all $ X =(x_1,...,x_{n-1}) \in A^{n-1}$ satisfying $\bar{\alpha}(X)=X$, we define $ad_{k,X} : A \to A$ by:
\[ad_{k,X}(z) = [X,\alpha^k(z)] = [x_1,...,x_{n-1},z], \quad \forall z \in A. \]
\end{definition}

\begin{lemma}[\cite{sheng:hom rep,n-ary hom rep}]
The map $ad_{k,X}$ defined above is an $\alpha^{k+1}$-derivation, which is called inner $\alpha^{k+1}$-derivation.
\end{lemma}
Denote by $Der_k(A)$ the set of $\alpha^k$-derivations of $A$ and by $Inn_k(A)$ the vector space spanned by the inner $\alpha^k$-derivation. Define $Der(A)$ and $Inn(A)$ by:
\[ Der(A)=\bigoplus_{k\geq 0} Der_k(A) \text{ and } Inn(A) = \bigoplus_{k\geq 0} Inn_k(A).  \]
Then we have:
\begin{lemma}[\cite{sheng:hom rep,n-ary hom rep}]
The space $Der(A)$ equipped with the commutator ($[D,D']=D\circ D' - D' \circ D$, for all $D,D' \in Der(A)$) is a Lie algebra, and $Inn(A)$ is an ideal of it.
\end{lemma}

%%%text
Now we will introduce two cohomology complexes for $n$-Hom-Lie algebras, the first one is relevant for the study of central extensions, while the second one is used for deformations. See \cite{n-ary hom rep} and \cite{sheng:hom rep} for reference and more details.
\begin{definition}[\cite{sheng:hom rep,n-ary hom rep}]
Let $(A,\bracket{\cdot,...,\cdot},\alpha)$ be a multiplicative $n$-Hom-Lie algebra. A $\K$-valued $p$-cochain is a linear map $\varphi : (\wedge^{n-1}A)^{\otimes p-1}\wedge A \to \K$.

 We define the coboundary operator for these cochains by :
\begin{align*}
d^p\varphi(X_1,...,X_p,z)&= \sum_{j=1}^p \sum_{k=j+1}^p (-1)^j \varphi \left( \bar{\alpha}(X_1),...,\widehat{X_j},..., [X_j , X_k]_\alpha, ..., \bar{\alpha}(X_p), \alpha(z) \right)\\
&+ \sum_{j=1}^p (-1)^j \varphi \left( \bar{\alpha}(X_1),...,\widehat{X_j},...,\bar{\alpha}(X_p),X_j \cdot z \right). 
\end{align*}
This cohomology complex is called scalar cohomology complex.
\end{definition}

\begin{definition}[\cite{sheng:hom rep,n-ary hom rep}]
Let $(A,\bracket{\cdot,...,\cdot},\alpha)$ be a multiplicative $n$-Hom-Lie algebra. An $A$-valued $p$-cochain is a linear map $\varphi : (\wedge^{n-1}A)^{\otimes p-1}\wedge A \to A$ such that \[ \alpha\para{\varphi\para{x_1,...,x_n}} = \varphi\para{\alpha(x_1),...,\alpha(x_n)}. \]

 The coboundary operator for these cochains, for $p \geq 2$, is given by :
\begin{align*}
&d^p\varphi(X_1,...,X_p,z) = \sum_{j=1}^p \sum_{k=j+1}^p (-1)^j \varphi \left( \bar{\alpha}(X_1),...,\widehat{X_j},..., [X_j , X_k]_\alpha, ..., \bar{\alpha}(X_p),\alpha(z) \right)\\
&+ \sum_{j=1}^p (-1)^j \varphi \left( \bar{\alpha}(X_1),...,\widehat{X_j},...,\bar{\alpha}(X_p),X_j \cdot z \right) 
+ \sum_{j=1}^p (-1)^{j+1} \alpha^{p-1}(X_j) \cdot \varphi \left( X_1,...,\widehat{X_j},...,X_p, z \right) \\
&+ (-1)^{p-1} \left( \varphi (X_1,...,X_{p-1}, \quad) \cdot_\alpha X_p \right) \cdot \alpha^{p-1}(z),
\end{align*}
where \[ \varphi (X_1,...,X_{p-1}, \quad) \cdot_\alpha X_p = \sum_{i=1}^{n-1} \left(\alpha(x_p^1),...,\alpha(x_p^{i-1}),\varphi (X_1,...,X_{p-1},x_p^i), ...,\alpha(x_p^{n-1})\right). \]

This cohomology complex is called adjoint cohomology complex.
\end{definition}

The elements of $Z^p(A,A) = \ker d^p$ (resp. $Z^p(A,\mathbb{K})$) are called $p$-cocycles, those of $B^p(A,A)= \operatorname{Im} d^{p-1}$ (resp. $B^p(A,\mathbb{K})$) are called coboundaries. The quotient $H^p=\frac{Z^p}{B^p}$ is the $p$-th cohomology group. We sometimes add in subscript the representation used in the cohomology complex, for example $Z_{ad}^p(A,A)$ denotes the set of $p$-cocycle for the adjoint cohomology and $Z_{0}^p(A,\mathbb{K})$  denotes the set of $p$-cocycle for the scalar cohomology.

%%%%%%%%%%%%%%%%%%%%%%%%%%%%

\subsection{$(n+1)$-Hom-Lie algebras induced by $n$-Hom-Lie algebras}

In \cite{ams:ternary} and \cite{ams:n}, the authors introduced a construction of a $3$-Hom-Lie algebra from a Hom-Lie algebra, and more generally an $(n+1)$-Hom-Lie algebra from an $n$-Hom-Lie algebra. It is called $(n+1)$-Hom-Lie algebra induced by $n$-Hom-Lie algebra, we recall here its definition and basic properties, and we look at the multiplicative case, we also give some more basic properties of these algebras.

\begin{definition}[\cite{ams:ternary, ams:n}]\label{def:phitau}
  Let $\phi:A^n\to A$ be an $n$-linear map and  $\tau : A \to \K$ be a linear form. Define $\phi_\tau:A^{n+1}\to A$ by
  \begin{align}
    \phi_\tau(x_1,...,x_{n+1}) = \sum_{k=1}^{n+1}(-1)^{k-1}\tau(x_k)\phi(x_1,...,\hat{x}_k,...,x_{n+1}),
  \end{align}
where the hat over $\hat{x}_k$ on the right hand side means that $x_{k}$ is excluded, that is 
 $\phi$ is calculated on $(x_1,..., x_{k-1}, x_{k+1}, ..., x_{n+1})$. 
\end{definition}
We will not be concerned with just any linear map $\tau$, but rather maps that have a generalized trace property. Namely

\begin{definition}[\cite{ams:ternary, ams:n}]\label{def:phitrace}
  For $\phi:A^n\to A$ we call a linear map $\tau:A\to\K$ a
  \emph{$\phi$-trace} if $\tau\para{\phi(x_1,...,x_n)}=0$ for all
  $x_1,...,x_n\in A$.
\end{definition}

\begin{lemma}[\cite{ams:ternary, ams:n}]\label{lemma:phitlinantisym}
Let $\phi:A^n\to A$ be a skew-symmetric $n$-linear map and $\tau$ a linear map $V\to\K$. Then $\phi_\tau$ is an $(n+1)$-linear skew-symmetric map. Furthermore, if $\tau$ is a $\phi$-trace then $\tau$ is a $\phi_\tau$-trace.
\end{lemma}

\begin{theorem}[\cite{ams:ternary, ams:n}]\label{thm:inducedhomnambulie}
Let $(A,\phi,\alpha_1,...,\alpha_{n-1})$ be an $n$-Hom-Lie algebra, $\tau$ a $\phi$-trace and $\alpha_n:A\to A$ a linear map. If it holds that
\begin{align}
 &\tau\para{\alpha_i(x)}\tau(y)=\tau(x)\tau\para{\alpha_i(y)}\label{eq:tautaualpha}\\
 &\tau\para{\alpha_i(x)}\alpha_j(y)=\alpha_i(x)\tau\para{\alpha_j(y)}\label{eq:taualphaalpha}
\end{align}
for all $i,j\in\{1,...,n\}$ and all $x,y\in A$, then $(A,\phi_\tau,\alpha_1,...,\alpha_n)$ is an $(n+1)$-Hom-Lie algebra. We shall say that $(A,\phi_\tau,\alpha_1,...,\alpha_n)$ is \emph{induced by $(A,\phi,\alpha_1,...,\alpha_{n-1})$}. We refer to $A$ when considering the given $n$-Hom-Lie algebra and $A_\tau$ when considering the induced $(n+1)$-Hom-Lie algebra.
\end{theorem}

Now we look at the case of multiplicative $n$-Hom-Lie algebras, and provide a new construction theorem. We start with the following lemma:

\begin{lemma} \label{lemma:DSG}
Let $A,B$ be $\K$-vector spaces, $\phi : A^n \to A$ and $\psi : A^n \to B$ $n$-linear skew-symmetric maps and $\tau : A \to \K$ a linear form. Then we have, for all $x_1,...,x_n,y_1,...,y_{n+1}\in A$:
\[ S=\sum_{k=1}^{n+1} \sum_{j=1,j\neq k}^{n+1} (-1)^{j-1} \tau(y_j) \tau(y_k)  \psi\para{y_1,...,\widehat{y_j},...,y_{k-1},\phi\para{x_1,...,x_n},...,y_{n+1} } = 0,\]
where $\widehat{y_k}$ means that $y_k$ is omited.
\end{lemma}

\begin{proof}
\begin{align*}
%%%%%%%%%%%%%%%
S &= \sum_{k=1}^{n+1} \sum_{j=1}^{k-1} (-1)^{j-1} \tau(y_j) \tau(y_k)  \psi\para{y_1,...,\widehat{y_j},...,y_{k-1},\phi\para{x_1,...,x_n},...,y_{n+1}} \\
&+  \sum_{k=1}^{n+1} \sum_{j=k+1}^{n+1} (-1)^{j-1} \tau(y_j) \tau(y_k)  \psi\para{y_1,...,\widehat{y_j},...,y_{k-1},\phi\para{x_1,...,x_n},...,y_{n+1}} \\
%%%%%%%%%%%%%%%%
&=  \sum_{k=1}^{n+1} \sum_{j=1}^{k-1} (-1)^{j-1} \tau(y_j) \tau(y_k)  \psi\para{y_1,...,\widehat{y_j},...,y_{k-1},\phi\para{x_1,...,x_n},...,y_{n+1}} \\
&+  \sum_{k=1}^{n+1} \sum_{j=1}^{k-1} (-1)^{k-1} \tau(y_j) \tau(y_k)  \psi\para{y_1,...,\widehat{y_k},...,y_{j-1},\phi\para{x_1,...,x_n},...,y_{n+1}} \\
%%%%%%%%%%%%%%%
&=  \sum_{k=1}^{n+1} \sum_{j=1}^{k-1} (-1)^{j-1} \tau(y_j) \tau(y_k)  \psi\para{y_1,...,\widehat{y_j},...,y_{k-1},\phi\para{x_1,...,x_n},...,y_{n+1}} \\
&+  \sum_{k=1}^{n+1} \sum_{j=1}^{k-1} (-1)^{j} \tau(y_j) \tau(y_k)  \psi\para{y_1,...,\widehat{y_j},...,y_{k-1},\phi\para{x_1,...,x_n},...,y_{n+1}}\\
%%%%%%%%%
&= 0
\end{align*}
\end{proof}

Now we give conditions under which an $n$-Hom-Lie algebras morphism is still a morphism of the induced $(n+1)$-Hom-Lie algebras:
\begin{proposition}\label{morph_tau}
Let $(A,\phi,\alpha_1,...,\alpha_{n-1})$ and $(B,\psi,\beta_1,...,\beta_{n-1})$ be $n$-Hom-Lie algebras. Let $\tau$ (resp. $\sigma$) be a $\phi$-trace (resp. $\psi$-trace) and $\alpha_n$ (resp. $\beta_n$) a linear map $\alpha :A\to A$ (resp. $\beta_n : B \to B$). Set $(A,\phi_\tau,\alpha_1,...,\alpha_n)$ (resp. $(B,\psi_\sigma,\beta_1,...,\beta_n)$) to be the induced $(n+1)$-Hom-Lie algebra. Let $f : A \to B$ be an $n$-Hom-Lie algebra homomorphism satisfying $\sigma \circ f = \tau$ and $f\circ \alpha_n = \beta_n \circ f$, then $f$ is an $(n+1)-$Hom-Lie algebra homomorphism of the induced algebras.
\end{proposition}

\begin{proof}
For all $x_1,...,x_{n+1} \in A$, we have:
\begin{align*}
f\para{\phi_\tau\para{x_1,...,x_{n+1}}} &= \sum_{i=1}^{n+1} (-1)^{i-1} \tau(x_i)f\para{\phi\para{x_1,...,x_{i-1},x_{i+1},...,x_{n+1}}}\\
&=  \sum_{i=1}^{n+1} (-1)^{i-1} \tau(x_i)\psi\para{f(x_1),...,f(x_{i-1}),f(x_{i+1}),...,f(x_{n+1})}\\
&=  \sum_{i=1}^{n+1} (-1)^{i-1} \sigma(f(x_i))\psi\para{f(x_1),...,f(x_{i-1}),f(x_{i+1}),...,f(x_{n+1})}\\
&=  \psi_\sigma\para{f(x_1),...,f(x_{n+1})}.
\end{align*}
We have $f\circ \alpha_i = \beta_i \circ f, \forall 1\leq i \leq n-1,$
because $f$ is an $n$-Hom-Lie algebra homomorphism, and we also have
$f\circ \alpha_n = \beta_n \circ f.$
This means that $f$ is  an $(n+1)$-Hom-Lie homomorphism of $A_\tau$ and $B_\sigma$.
\end{proof}

The new theorem for constructing $(n+1)$-Hom-Lie algebras induced by $n$-Hom-Lie algebras can be formulated as follows:

\begin{theorem}\label{thm:inducedmul}
Let $(A,\phi,\alpha)$ be an $n$-Hom-Lie algebra and $\tau$ a $\phi$-trace. If $\tau \circ \alpha = \tau$ then $(A,\phi_\tau,\alpha_1,...,\alpha_{n})$ is an $(n+1)$-Hom-Lie algebra. Moreover, if $(A,\phi,\alpha)$ is a multiplicative $n$-Hom-Lie algebra, then, under the same condition, $(A,\phi_\tau,\alpha)$ is a multiplicative $(n+1)$-Hom-Lie algebra.
\end{theorem}

\begin{proof}
We know, by Lemma \ref{lemma:phitlinantisym}, that $\phi_\tau$ is an $(n+1)$-linear skew-symmetric map. We show that it satisfies the Hom-Nambu identity. Let $L$ be its left-hand side, and $R$ its right-hand side:
\begin{small}
\begin{align*}
L &= \phi_\tau\para{\alpha(x_1),...,\alpha(x_n),\phi_\tau\para{y_1,...,y_{n+1}}}\\
&= \sum_{i=1}^n (-1)^{i-1} \tau(\alpha(x_i)) \phi\para{\alpha(x_1),...,\widehat{\alpha(x_i)},...,\alpha(x_n),\phi_\tau\para{y_1,...,y_{n+1}}}\\
&= \sum_{i=1}^n \sum_{j=1}^{n+1}  (-1)^{i+j} \tau(\alpha(x_i))\tau(y_j) \phi\para{\alpha(x_1),...,\widehat{\alpha(x_i)},...,\alpha(x_n),\phi\para{y_1,...,,\widehat{y_j},...,y_{n+1}}}\\
&= \sum_{i=1}^n \sum_{j=1}^{n+1} \sum_{k=1 ; k\neq j}^{n+1}  (-1)^{i+j} \tau(\alpha(x_i))\tau(y_j)  \phi\para{\alpha(y_1),...,\widehat{\alpha(y_j)},...,\phi\para{x_1,...,\widehat{x_i},...,x_n,y_k},...,\alpha(y_{n+1})}\\
&= \sum_{i=1}^n \sum_{j=1}^{n+1} \sum_{k=1 ; k\neq j}^{n+1}  (-1)^{i+j} \tau(x_i)\tau(y_j)  \phi\para{\alpha(y_1),...,\widehat{\alpha(y_j)},...,\phi\para{x_1,...,\widehat{x_i},...,x_n,y_k},...,\alpha(y_{n+1})}
\end{align*}

\begin{align*}
R &= \sum_{k=1}^{n+1}  \phi_\tau\para{\alpha(y_1),...,\phi_\tau\para{x_1,...,x_n,y_k},...,\alpha(y_{n+1})} \\
&=\sum_{k=1}^{n+1} \sum_{j=1,j\neq k}^{n+1} (-1)^{j-1} \tau(\alpha(y_j))  \phi\para{\alpha(y_1),...,\widehat{\alpha(y_j)},...,\phi_\tau\para{x_1,...,x_n,y_k},...,\alpha(y_{n+1})} \\
&=\sum_{k=1}^{n+1} \sum_{j=1,j\neq k}^{n+1} \sum_{i=1}^n (-1)^{i+j} \tau(\alpha(y_j)) \tau(x_i) \phi\para{\alpha(y_1),...,\widehat{\alpha(y_j)},...,\phi\para{x_1,...,\widehat{x_i},...,x_n,y_k},...,\alpha(y_{n+1})} \\
&+ (-1)^n \sum_{k=1}^{n+1} \sum_{j=1,j\neq k}^{n+1} (-1)^{j-1} \tau(\alpha(y_j)) \tau(y_k)  \phi\para{\alpha(y_1),...,\widehat{\alpha(y_j)},...,\phi\para{x_1,...,x_n},...,\alpha(y_{n+1})} \\
&= \sum_{i=1}^n \sum_{j=1}^{n+1} \sum_{k=1,k\neq j}^{n+1}  (-1)^{i+j} \tau(y_j) \tau(x_i) \phi\para{\alpha(y_1),...,\widehat{\alpha(y_j)},...,\phi\para{x_1,...,\widehat{x_i},...,x_n,y_k},...,\alpha(y_{n+1})} \\
&+ (-1)^n \sum_{k=1}^{n+1} \sum_{j=1,j\neq k}^{n+1} (-1)^{j-1} \tau(y_j) \tau(y_k)  \phi\para{\alpha(y_1),...,\widehat{\alpha(y_j)},...,\alpha(y_{k-1}),\phi\para{x_1,...,x_n},...,\alpha(y_{n+1})} \\
&= L + (-1)^n \sum_{k=1}^{n+1} \sum_{j=1,j\neq k}^{n+1} (-1)^{j-1} \tau(y_j) \tau(y_k)  \phi\para{\alpha(y_1),...,\widehat{\alpha(y_j)},...,\alpha(y_{k-1}),\phi\para{x_1,...,x_n},...,\alpha(y_{n+1})} \\
&= L. \qquad \text{ (by Lemma \ref{lemma:DSG})}
\end{align*}
\end{small}
We also have, in the case of a multiplicative algebra, by Proposition \ref{morph_tau} that if $\alpha$ is an endomorphism of $(A,\phi,\alpha)$ then $\alpha$ is an endomorphism of $(A,\phi_\tau,\alpha)$. That is $(A,\phi_\tau,\alpha)$ is multiplicative.
\end{proof}

\begin{remark}
This construction is not a particular case of Theorem \ref{thm:inducedhomnambulie}. Indeed, if we assume that $\alpha_1=...=\alpha_n=\alpha$ and the hypotheses of Theorem \ref{thm:inducedmul}, conditions \ref{eq:tautaualpha} and \ref{eq:taualphaalpha} are equivalent to $\alpha\para{\tau(x)y}=\alpha\para{\tau(y)x}$, which forces the algebra to be $1$-dimensional in case $\alpha$ is injective.
\end{remark}

The following proposition shows that the algebra obtained by applying Theorem \ref{thm:inducedmul} and Proposition \ref{twist}, when possible, in any order is the same.

\begin{proposition}
Let $\para{A,\bracket{\cdot,...,\cdot}}$ be an $n$-Lie algebra, $\alpha : A \to A$ an algebra morphism and $\tau : A \to \K$ such that $\para{A,\bracket{\cdot,...,\cdot}_\alpha,\alpha}$ (Proposition \ref{twist}) and $\tau$ satisfies conditions of Theorem \ref{thm:inducedmul}. Then we have $\bracket{\cdot,...,\cdot}_{\alpha,\tau} = \bracket{\cdot,...,\cdot}_{\tau,\alpha}.$
\end{proposition}

\begin{proof}
Let $x_1,...,x_{n+1} \in A$, we have:
\begin{align*}
\bracket{x_1,...,x_{n+1}}_{\alpha,\tau}&= \sum_{i=1}^{n+1} (-1)^{i-1} \tau(x_i) \bracket{x_1,...,\widehat{x_i},...,x_{n+1}}_\alpha \\
&
= \sum_{i=1}^{n+1} (-1)^{i-1} \tau(x_i) \alpha \para{\bracket{x_1,...,\widehat{x_i},...,x_{n+1}}} \\
&=\alpha \para{\sum_{i=1}^{n+1} (-1)^{i-1} \tau(x_i) \bracket{x_1,...,\widehat{x_i},...,x_{n+1}}}\\
&=\alpha \para{\bracket{x_1,...,x_{n+1}}_\tau}
=\bracket{x_1,...,x_{n+1}}_{\tau,\alpha}.
\end{align*}

\end{proof}

Now we give two results about subalgebras and ideals of $(n+1)$-Hom-Lie algebras induced by $n$-Hom-Lie algebras.Let $(A,\bracket{\cdot,...,\cdot},\alpha_1,...,\alpha_{n-1})$ be an $n$-Hom-Lie algebra, $\tau$ a trace, $\alpha_n : A \to A$ a linear map and $(A,\bracket{\cdot,...,\cdot}_\tau, \alpha_1,...,\alpha_n)$ the induced $(n+1)$-Hom-Lie algebra.
\begin{proposition}
Let $B$ be a subalgebra of $A$. If $\alpha_n(B) \subseteq B$ then $B$ is also a subalgebra of $A_\tau$.
\end{proposition}
\begin{proof}
Let $B$ be a subalgebra of $(A,\bracket{\cdot,...,\cdot},\alpha_1,...,\alpha_{n-1})$. We have that $\alpha_i(B) \subseteq B$ for $1 \leq i \leq n-1$ because $B$ is a subalgebra of A, and $\alpha_n(B) \subseteq B$, then $\alpha_i(B) \subseteq B$ for $1 \leq i \leq n$.

Now let $x_1,...,x_{n+1} \in B$:
\[ \bracket{x_1,...,x_{n+1}}_\tau = \sum_{i=1}^{n+1} (-1)^{i-1}\tau(x_i)\bracket{x_1,...,\hat{x}_{i},...,x_{n+1}},\]
which is a linear combination of elements of $B$ and then belongs to $B$.
\end{proof}

\begin{proposition}
Let $J$ be an ideal of  $A$. If $\alpha_n(J)\subseteq J$, then $J$ is an ideal of $A_\tau$ if and only if
\[   \bracket{A,...,A} \subseteq J \text{ or } J \subseteq \ker \tau. \]
\end{proposition}
\begin{proof}
Let $J$ be an ideal of $A$, satisfying $\alpha_n(J)\subseteq J$, we have that $\alpha_i(J)\subseteq J$ for $1 \leq i \leq n$. Consider $j \in J$ and $x_1,...,x_n \in A$, then we have:
\[ \bracket{x_1,...,x_n,j}_\tau = \sum_{i=1}^n (-1)^{i-1} \tau(x_i) \bracket{x_1,...,\hat{x}_i,...,x_n,j} + (-1)^n \tau(j) \bracket{x_1,...,x_n}. \]
We have $\sum_{i=1}^n (-1)^{i-1} \tau(x_i) \bracket{x_1,...,\hat{x}_i,...,x_n,j} \in J$, then, to obtain $\left. \bracket{x_1,...,x_n,j}_\tau \in J \right.$ it is necessary and sufficient to have $\tau(j) \bracket{x_1,...,x_n} \in J$, which is equivalent to $\tau(j) = 0$ or $\bracket{x_1,...,x_n} \in J$. %%%!!!
\end{proof}

%%%%%%%%%%%%%%%%%%%%%%%%%%%%%%%%%%%%%%%%%%%%%%%%%%%%%
%%%%%%%%%%%%%%%%%%%%%%%%%%%%%%%%%%%%%%%%%%%%%%%%%%%%%

\section{Solvability and nilpotency of $(n+1)$-Hom-Lie algebras induced by $n$-Hom-Lie algebras} \label{sec:solvnil}
In this section, we define the derived series, central descending series and center of $n$-Hom-Lie algebras, then we show the relations between central descending series, derived series and center of an $n$-Hom-Lie algebra, and those of the induced $(n+1)$-Hom-Lie algebra. The definitions generalize those given in \cite{Filippov:nLie}. The results generalize those presented in \cite{km:n-ary} and independently in \cite{Bai:n}.

%%%%%%%%%%%%%%%%%%%Subsection?%%%%%%%%%%%%%%%%%
\subsection{Solvability and nilpotency of $n$-Hom-Lie algebras}
Now, we define the derived series, central descending series and the center of an $n$-Hom-Lie algebra, these generalization are relevant only in the case of multiplicative algebras.
\begin{definition}
Let $(A,[\cdot , ... ,\cdot],\alpha)$ be a multiplicative $n$-Hom-Lie algebra, and $I$ an ideal of $A$. We define $D^r(I), r \in \mathbb{N}$, the derived series of $I$ by:
\[D^0(I)=I \text{ and }D^{r+1}(I)=[D^r(I),...,D^r(I)].\]
\end{definition}
\begin{proposition}
The subspaces $D^r(I), r \in \mathbb{N}$, are subalgebras of $(A,[\cdot , ... ,\cdot],\alpha)$.
\end{proposition}

\begin{proof}
We proceed by induction over $r \in \mathbb{N}$, the case of $r=0$ is trivial. Now suppose that $D^r(I)$ is a subalgebra of $A$, we prove that $D^{r+1}(I)$ is a subalgebra of $A$:
\begin{itemize}
\item Let $y \in D^{r+1}(I)$
\begin{align*}
\alpha(y)&=\alpha([y_{1},...,y_{n}])= [\alpha(y_{1}),...,\alpha(y_{n})] \quad y_{1},...,y_{n} \in D^r(I),
\end{align*}
which is in $D^{r+1}(I)$ because $\alpha(y_{1}),...,\alpha(y_{n})\in D^r(I)$. That is $\alpha(D^{r+1}(I))\subseteq D^{r+1}(I)$.
\item Let $x_1,...,x_n \in D^{r+1}(I)$:
\begin{align*}
[x_1,...,x_n] &= [[x_{11},...,x_{1n}],...,[x_{n1},...,x_{nn}]] \quad x_{11},...,x_{1n},...,x_{n1},...,x_{nn} \in D^r(I) \\
[x_1,...,x_n]&\in D^{r+1}(I).
\end{align*}
\end{itemize}
\end{proof}
\begin{proposition}
Let $(A,[\cdot , ... ,\cdot],\alpha)$ be a multiplicative $n$-Hom-Lie algebra, and $I$ an ideal of $A$. If $\alpha$ is surjective then $D^r(I),\, r \in \mathbb{N},$ are ideals of $A$.
\end{proposition}
\begin{proof}
We already have that $D^r(I), \, r \in \mathbb{N},$ are subalgebras, we only need to prove that for all $x_1,...,x_{n-1} \in A$, and $y \in D^r(I)$, $[x_1,...,x_{n-1},y]\in  D^r(I)$.

We proceed by induction over $r \in \mathbb{N}$, the case of $r=0$ is trivial. Now suppose that $D^r(I)$ is an ideal of $A$, we prove that $D^{r+1}(I)$ is an ideal of $A$:

Let $x_1,...,x_{n-1} \in A$ and $y\in D^{r+1}(I)$:
\begin{align*}
[x_1,...,x_{n-1},y] &= [x_1,...,x_{n-1},[y_1,...,y_n]] \quad y_1,...,y_n \in D^r(I) \\
&= [\alpha(v_1),...,\alpha(v_{n-1}),[y_1,...,y_n]] \text{ for some } v_1,...,v_{n-1} \in A \\
&= \sum_{i=1}^n [\alpha(y_1),...,\alpha(y_{i-1}),[v_1,...,v_{n-1},y_i], \alpha(y_{i+1}),...,\alpha(y_{n})]\\
[x_1,...,x_{n-1},y]& \in D^{r+1}(I),
\end{align*}
because all the $\alpha(y_{i}) \in D^{r}(I)$ and all the $[v_1,...,v_{n-1},y_i]\in D^{r}(I)$ ($D^r(I)$ is an ideal), and then all the $[\alpha(y_1),...,\alpha(y_{i-1}),[v_1,...,v_{n-1},y_i], \alpha(y_{i+1}),...,\alpha(y_{n})]$ are in  $D^{r+1}(I)$. 
\end{proof}

\begin{definition}
Let $(A,[\cdot , ... ,\cdot],\alpha)$ be a multiplicative $n$-Hom-Lie algebra, and $I$ an ideal of $A$. We define $C^r(I) ,\, r \in \mathbb{N}$, the central descending series of $I$ by:
\[C^0(I)=I \text{ and }C^{r+1}(I)=[C^r(I),I ,...,I].\]
\end{definition}

\begin{proposition}
Let $(A,[\cdot , ... ,\cdot],\alpha)$ be a multiplicative $n$-Hom-Lie algebra, and $I$ an ideal of $A$. If $\alpha$ is surjective then $C^r(I) ,\, r \in \mathbb{N},$ are ideals of $A$.
\end{proposition}
\begin{proof}
We proceed by induction over $r \in \mathbb{N}$, the case of $r=0$ is trivial. Now suppose that $C^r(I)$ is an ideal of $A$, we prove that $D^{r+1}(I)$ is an ideal of $A$:
\begin{itemize}
\item Let $y \in C^{r+1}(I)$
\begin{align*}
\alpha(y)&=\alpha([y_{1},...,y_{n-1},w])= [\alpha(y_{1}),...,\alpha(y_{n-1}), \alpha(w)] \quad y_{1},...,y_{n-1} \in I, w \in C^r(I), 
\end{align*}
which is in $C^{r+1}(I)$ because $\alpha(y_{1}),...,\alpha(y_{n-1})\in I$ and $\alpha(w)\in C^r(I)$. \\That is $\alpha(C^{r+1}(I))\subseteq C^{r+1}(I)$.
\item Let $x_1,...,x_{n-1} \in A$ and $y\in C^{r+1}(I)$:
\begin{align*}
[x_1,...,x_{n-1},y] &= [x_1,...,x_{n-1},[y_{1},...,y_{n-1},w]] \quad y_1,...,y_{n-1} \in I, w \in C^r(I) \\
&= [\alpha(v_1),...,\alpha(v_{n-1}),[y_1,...,y_{n-1},w]] \text{ for some } v_1,...,v_{n-1} \in A \\
&= \sum_{i=1}^{n-1} [\alpha(y_1),...,\alpha(y_{i-1}),[v_1,...,v_{n-1},y_i], \alpha(y_{i+1}),...,\alpha(y_{n-1}),\alpha(w)] \\
 &+ [\alpha(y_1),...,\alpha(y_{n-1}),[v_1,...,v_{n-1},w]],
\end{align*}
which is in $C^{r+1}(I)$, because all the $\alpha(y_i)\in I$, $\alpha(w) \in C^{r}(I)$, all the $[v_1,...,v_{n-1},y_i]\in I$ (I is an ideal) and $[v_1,...,v_{n-1},w]\in C^{r}(I)$ ($C^r(I)$ is an ideal). Therefore all the $[\alpha(y_1),...,\alpha(y_{i-1}),[v_1,...,v_{n-1},y_i], \alpha(y_{i+1}),...,\alpha(y_{n})]$ and\\ $[\alpha(y_1),...,\alpha(y_{n-1}),[v_1,...,v_{n-1},w]]$ are in  $C^{r+1}(I)$. 
\end{itemize}
\end{proof}

\begin{definition}
Let $(A,[\cdot , ... , \cdot],\alpha)$ be a multiplicative $n$-Hom-Lie algebra, and $I$ an ideal of $A$. The ideal $I$ is said to be solvable if there exists $r \in \mathbb{N}$ such that $D^r(I)=\{0\}$. It is said to be nilpotent if there exists $r \in \mathbb{N}$ such that $C^r(I)=\{0\}$.
\end{definition}

\begin{definition}
Let $(A,[\cdot , ... , \cdot],\alpha)$ be a multiplicative $n$-Hom-Lie algebra. define the center of $A$, denoted by $Z(A)$, as:
\[Z(A) = \{z \in A : [x_1,...,x_{n-1},z]=0, \forall x_1,...,x_{n-1} \in A\}.\]
\end{definition}

\begin{proposition}
Let $(A,[\cdot , ... , \cdot],\alpha)$ be a multiplicative $n$-Hom-Lie algebra. If $\alpha$ is surjective then the center of $A$ is an ideal of $A$.
\end{proposition}
\begin{proof}
Let $z \in Z(A)$ and $x_1,...,x_{n-1} \in A$, we put $x_i = \alpha(u_i)$, for $1 \leq i \leq n-1$, then we have:
\begin{align*}
\bracket{x_1,...,x_{n-1},\alpha(z)}= \bracket{\alpha(u_1),...,\alpha(u_{n-1}),\alpha(z)} 
= \alpha\para{\bracket{u_1,...,u_{n-1},z}} 
=0,
\end{align*}
that is $\alpha(Z(A)) \subseteq Z(A)$.

Let $z \in Z(A)$ and $x_1,...,x_{n-1},y_1,...,y_{n-1} \in A$:
\begin{align*}
\bracket{x_1,...,x_{n-1},\bracket{y_1,...,y_{n-1},z}}= \bracket{x_1,...,x_{n-1},0}
=0.
\end{align*}
Which means that $Z(A)$ is an ideal of $A$.
\end{proof}

\begin{remark}
In the preceding definitions, if $\alpha=Id_A$ then we find back the definitions introduced by Filippov \cite{Filippov:nLie} for $n$-Lie algebras. We have different definitions of derived series and central descending series of $n$-Lie algebras (See \cite{Kasymov:nLie}) which are not included in this generalization.
\end{remark}

\subsection{Solvability and nilpotency of $(n+1)$-Hom-Lie algebras induced by $n$-Hom-Lie algebras}
Now we show the relationships between central descending series, derived series and center of an $n$-Hom-Lie algebra, and those of the induced $(n+1)$-Hom-Lie algebra.

\begin{theorem} \label{solv2}%%%
Let $(A,\bracket{\cdot,...,\cdot},\alpha_1,...,\alpha_{n-1})$ be an $n$-Hom-Lie algebra, $\tau$ a trace, $\alpha_n : A \to A$ a linear map satisfying the conditions of Theorem \ref{thm:inducedhomnambulie} and $(A,\bracket{\cdot,...,\cdot}_\tau, \alpha_1,...,\alpha_n)$ the induced $(n+1)$-Hom-Lie algebra. The induced algebra  is solvable, more precisely $D^2(A_\tau) = 0$ i.e. $\left( D^1(A_\tau)=\bracket{A,...,A}_\tau,\bracket{\cdot,...,\cdot}_\tau\right)$ is abelian.
\end{theorem}
\begin{proof}
Let $x_1,...,x_{n+1}\in \bracket{A,...,A}_\tau$, $x_i=\bracket{x_i^1,...,x_i^{n+1}}_\tau$,  $\forall 1 \leq i \leq n+1$ , then:
\begin{align*}
&\bracket{x_1,...,x_{n+1}}_\tau \\ &= \sum_{i=1}^{n+1} \tau\para{\bracket{x_i^1,...,x_i^{n+1}}_\tau}\bracket{\bracket{x_1^1,...,x_1^{n+1}}_\tau,...,\widehat{\bracket{x_i^1,...,x_i^{n+1}}_\tau},...\bracket{x_{n+1}^1,...,x_{n+1}^{n+1}}_\tau}= 0, 
\end{align*}
because $\tau\left([\cdot, ... ,\cdot]_\tau \right)=0$.

\end{proof}

\begin{proposition}
Let $(A,\bracket{\cdot,...,\cdot},\alpha)$ be a multiplicative $n$-Hom-Lie algebra, $\tau$ a trace map satisfying $\tau \circ \alpha = \tau$ and $(A,\bracket{\cdot,...,\cdot}_\tau,\alpha)$ the induced $(n+1)$-Hom-Lie algebra. Let $c \in Z(A)$, if $\tau(c) = 0$ then $c \in Z(A_\tau)$. 
Moreover, if $A$ is not abelian then $\tau(c) = 0$ if and only if $c \in Z(A_\tau)$.
\end{proposition}

\begin{proof}
Let $c \in Z(A)$ and $x_1,...,x_n \in A$:
\begin{align*}
\bracket{x_1,...,x_n,c}_\tau &= \sum_{i=1}^n (-1)^{i-1} \tau(x_i)\bracket{x_1,...,\widehat{x_i},...,x_n,c} + (-1)^n \tau(c) \bracket{x_1,...,x_n}\\
&= (-1)^n \tau(c) \bracket{x_1,...,x_n}.
\end{align*}
If $\tau(c)=0$ then $c \in Z(A_\tau)$.\\ Conversely, if $c \in Z(A_\tau)$ and $A$ is not abelian, then $\tau(c)=0$. 
\end{proof}

\begin{proposition}
Let $(A,\bracket{\cdot,...,\cdot},\alpha)$ be a non-abelian multiplicative $n$-Hom-Lie algebra, $\tau$ a trace satisfying $\tau \circ \alpha = \tau$, and $(A,\bracket{\cdot,...,\cdot}_\tau,\alpha)$ the induced $(n+1)$-Hom-Lie algebra. If $\tau\para{Z(A)} \neq \{0\}$ then $A_\tau$ is not abelian.
\end{proposition}

\begin{proof}
Let $x_1,...,x_n \in A$ such that $\bracket{x_1,...,x_n} \neq 0$ and $c \in Z(A)$ such that $\tau(c) \neq 0$ then we have:
\begin{align*}
\bracket{x_1,...,x_n,c}_\tau &= \sum_{i=1}^n (-1)^{i-1} \tau(x_i)\bracket{x_1,...,\widehat{x_i},...,x_n,c} + (-1)^n \tau(c) \bracket{x_1,...,x_n}\\
&= (-1)^n \tau(c) \bracket{x_1,...,x_n} \neq 0,
\end{align*}
which means that $A_\tau$ is not abelian.
\end{proof}

\begin{proposition}
Let $\para{A,\bracket{\cdot,...,\cdot},\alpha}$ be a multiplicative $n$-Hom-Lie algebra, $\tau$ be a trace satisfying $\tau \circ \alpha = \tau$, and $\para{A,\bracket{\cdot,...,\cdot}_\tau,\alpha}$ the induced algebra. Let $\para{C^p(A)}_p$ be the central descending series of $A$, and $\para{C^p(A_\tau)}_p$ be the central descending series of $A_\tau$. Then we have 
\[ C^p(A_\tau) \subset C^p(A), \forall p \in \mathbb{N}. \]
If there exists $u \in A$ such that $\bracket{u,x_1,...,x_n}_\tau = \bracket{x_1,...,x_n}, \forall x_1,...,x_n \in A$,  then
\[ C^p(A_\tau) = C^p(A), \forall p \in \mathbb{N}. \]
\end{proposition}
\begin{proof}
Theorem \ref{thm:inducedmul} provides that $A_\tau$ is multiplicative. We proceed by induction over $p \in \mathbb{N}$. The case of $p=0$ is trivial, for $p=1$ we have:
\[\forall x = \bracket{x_1,...,x_{n+1}}_\tau \in C^1(A_\tau),\  x=\sum_{i=1}^{n+1} (-1)^{i-1} \tau(x_i) \bracket{x_1,...,\hat{x}_i,...,x_{n+1}}, \] which is a linear combination of elements of $C^1(A)$ and then is an element of $C^1(A)$.
Suppose now that there exists $u \in A$ such that $\bracket{u,x_1,...,x_n}_\tau = \bracket{x_1,...,x_n}, \forall x_1,...,x_n \in A$. Then for $x = \bracket{x_1,...,x_n} \in C^1(A)$, $x=\bracket{u,x_1,...,x_n}_\tau$ and hence it  is an element of $C^1(A_\tau)$.

Now, we suppose this proposition  true for some $p \in \mathbb{N}$, and let $x \in C^{p+1}(A_\tau)$. Then $x=\bracket{a,x_1,...,x_n}_\tau$ with $x_1,...,x_n \in A$ and $a \in C^{p}(A_\tau)$,
\[x=\bracket{a,x_1,...,x_n}_\tau = \sum_{i=1}^{n} (-1)^{i} \tau(x_i) \bracket{a,x_1,...,\hat{x}_i,...,x_n}, \qquad (\tau(a)=0)\]
which is an element of $C^{p+1}(A)$ because $a \in C^{p}(A_\tau) \subset C^p(A)$.
Assume there exists $u \in A$ such that $\bracket{u,x_1,...,x_n}_\tau = \bracket{x_1,...,x_n}, \forall x_1,...,x_n \in A$. Then if $x \in C^{p+1}(A)$ we have $x = \bracket{a,x_1,...,x_{n-1}}$ with $a \in C^{p}(A)$ and $x_1,...,x_{n-1} \in A$. Therefore \[ x = \bracket{a,x_1,...,x_{n-1}} = \bracket{u,a,x_1,...,x_{n-1}}_\tau = (-1)^n \bracket{a,x_1,...,x_{n-1},i}_\tau \in C^{p+1}(A_\tau).\]  
\end{proof}

\begin{remark} \label{Hom-D3subD}
It also results from the preceding proposition that \[D^1(A_\tau) = \bracket{A,...,A}_\tau \subset D^1(A) = \bracket{A,...,A},\] and if there exists $u \in A$ such that \[\bracket{u,x_1,...,x_n}_\tau = \bracket{x_1,...,x_n}, \forall x_1,...,x_n \in A.\] Then $D^1 (A_\tau) = D^1(A)$. For the rest of the derived series, we have obviously the first inclusion by Theorem \ref{solv2}, which states also that all induced algebras are solvable.
\end{remark}

\begin{theorem}
Let $\para{A,\bracket{\cdot,...,\cdot},\alpha}$ be a multiplicative $n$-Hom-Lie algebra, $\tau$ be a trace satisfying $\tau \circ \alpha = \tau$, and $\para{A,\bracket{\cdot,...,\cdot}_\tau,\alpha}$ the induced algebra. Then,
if $A$ is nilpotent of class $p$, we have $A_\tau$ is nilpotent of class at most  $p$.
Moreover, if there exists $u \in A$ such that $\bracket{u,x_1,...,x_n}_\tau = \bracket{x_1,...,x_n}, \forall x_1,...,x_n \in A$, then $A$ is nilpotent of class  $p$ if and only if $A_\tau$ is nilpotent of class  $p$.
\end{theorem}
\begin{proof} 
Theorem \ref{thm:inducedmul} provides that $A_\tau$ is multiplicative.
\begin{enumerate}
\item Suppose that $\para{A,\bracket{\cdot ,..., \cdot}}$ is nilpotent of class $p\in \mathbb{N}$, then $C^p(A)=\{0\}$. By the preceding proposition, $C^p (A_\tau) \subseteq C^p(A)=\{0\}$, therefore $\para{A,\bracket{\cdot ,..., \cdot}_\tau}$ is nilpotent of class at most $p$.
\item We suppose now that $\para{A,\bracket{\cdot ,..., \cdot}_\tau}$ is nilpotent of class $p \in \mathbb{N}$, and that there exists $u \in A$ such that $\bracket{u,x_1,...,x_n}_\tau = \bracket{x_1,...,x_n}, \forall x_1,...,x_n \in A$, then $C^p(A_\tau)=\{0\}$. By the preceding proposition, $C^p(A) = C^p(A_\tau)=\{0\}$. Therefore $\para{A,\bracket{\cdot ,..., \cdot}}$ is nilpotent, since $C^{p-1}(A) = C^{p-1}(A_\tau) \neq \{0\}$,  $\para{A,\bracket{\cdot ,..., \cdot}_\tau}$ and $\para{A,\bracket{\cdot ,..., \cdot}}$ have the same nilpotency class.
\end{enumerate}
\end{proof}

%%%%%%%%%%%%%%%%%%%%%%%%%%%%%%%%%%%%%%%%%%%%%
%%%%%%%%%%%%%%%%%%%%%%%%%%%%%%%%%%%%%%%%%%%%%

\section{Central extensions of $(n+1)$-Hom-Lie algebras induced by $n$-Hom-Lie algebras} \label{sec:centralextentions}
We first review  the definition and the main properties of central extensions of $n$-Hom-Lie algebras, then we study central extensions of $(n+1)$-Hom-Lie algebras induced by $n$-Hom-Lie algebras. Central extensions of Hom-Lie algebras were introduced in \cite{sheng:hom rep} together with the relevant cohomology complex, those definitions were then generalized to the $n$-ary case in \cite{n-ary hom rep}.

\begin{definition}[\cite{sheng:hom rep,n-ary hom rep}]
Let $\para{A,\bracket{\cdot,...,\cdot},\alpha}$ be a multiplicative $n$-Hom-Lie algebra. We call central extension of $A$ the space $\bar{A}=A \oplus \K c$ equipped with the bracket $\bracket{\cdot,...,\cdot}_c$ and the morphism $\alpha_c$ defined by:
\[ \bracket{x_1,...,x_n}_c = \bracket{x_1,...,x_n} + \omega\para{x_1,...,x_n} c \text{ and } \bracket{x_1,...,x_{n-1},c}=0 , \forall x_1,...,x_n \in A. \]
\[ \alpha_c(\bar{x}) = \alpha(x) + (\lambda (\bar{x})) c \quad \forall \bar{x}=x+x_c c \in \bar{A}, x\in A \]
Where $\lambda: \bar{A}\to \K$ is a linear map and $\omega: A^n \to \K$ is a skew-symmetric $n$-linear map such that $\bracket{\cdot,...,\cdot}_c$ and $\alpha_c$ satisfy the Hom-Nambu identity.
\end{definition}

\begin{proposition}[\cite{sheng:hom rep,n-ary hom rep}]
\begin{itemize}
\item The bracket of a central extension of an $n$-Hom-Lie algebra satisfies the Hom-Nambu identity if and only if the map $\omega$ is a $2$-cocycle for the scalar cohomology of $n$-Hom-Lie algebras.
\item Two central extensions defined by two maps $\omega_1$ and $\omega_2$ are isomorphic if and only if $\omega_2 - \omega_1$ is a $2$-coboundary for the scalar cohomology of $n$-Hom-Lie algebras.
\end{itemize}
\end{proposition}

Now we show the relationship between the central extensions of an $n$-Hom-Lie algebra and those of the induced $(n+1)$-Hom-Lie algebra (by some trace $\tau$): %%%%!!!new
\begin{theorem}
Let $(A,\bracket{\cdot,...,\cdot},\alpha)$ be a multiplicative $n$-Hom-Lie algebra, $\tau$ be a trace satisfying $\tau \circ \alpha = \tau$, and $\para{A,\bracket{\cdot,...,\cdot}_\tau,\alpha}$ be the induced (multiplicative) $(n+1)$-Hom-Lie algebra. Let $\para{\bar{A},\bracket{\cdot,...,\cdot}_c,\alpha_c}$ be a central extension of $(A,\bracket{\cdot,...,\cdot},\alpha)$, where 
\[\bar{A}=A\oplus \mathbb{K} c ,\quad \bracket{x_1,...,x_n}_c = \bracket{x_1,...,x_n} + \omega\para{x_1,...,x_n}c\text{ and }\alpha_c(\bar{x})=\alpha(x)+\lambda(\bar{x}) c, \] 
with $\lambda : \bar{A} \to \K$, and assume that $\tau$ extends  to $\bar{A}$ by $\tau(c)=0$. Then the $(n+1)$-Hom-Lie algebra $\para{\bar{A},\bracket{\cdot,...,\cdot}_{c,\tau},\alpha_c}$ induced by $\para{\bar{A},\bracket{\cdot,...,\cdot}_c,\alpha_c}$, is a central extension of $(A,\bracket{\cdot,...,\cdot}_\tau,\alpha)$, where 
\[\bracket{x_1,...,x_n}_{c,\tau} = \bracket{x_1,...,x_{n+1}}_\tau + \omega_\tau \para{x_1,...,x_{n+1}}c\]  with \[\omega_\tau \para{x_1,...,x_{n+1}} = \sum_{i=1}^{n+1} (-1)^{i-1} \tau\para{x_i} \omega(x_1,...,\hat{x}_i,...,x_{n+1}). \]
\end{theorem} 
\begin{proof}
We consider the algebra $\para{\bar{A},\bracket{\cdot,...,\cdot}_{c,\tau},\alpha_c}$ induced by $\para{\bar{A},\bracket{\cdot,...,\cdot}_{c},\alpha_c}$.

Let $x_1,...,x_{n+1} \in A$:
\begin{align*}
&\bracket{x_1,...,x_{n+1}}_{c,\tau} = \sum_{i=1}^{n+1} (-1)^{i-1} \tau\para{x_i}\bracket{x_1,...,\hat{x}_i,...,x_{n+1}}_c \\
&=  \sum_{i=1}^{n+1} (-1)^{i-1} \tau\para{x_i}\left( \bracket{x_1,...,\hat{x}_i,...,x_{n+1}} +  \omega(x_1,...,\hat{x}_i,...,x_{n+1})c \right) \\
&=  \sum_{i=1}^{n+1} (-1)^{i-1} \tau\para{x_i} \bracket{x_1,...,\hat{x}_i,...,x_{n+1}}    
+  \left( \sum_{i=1}^{n+1} (-1)^{i-1} \tau\para{x_i} \omega(x_1,...,\hat{x}_i,...,x_{n+1}) \right) c \\
&= \bracket{x_1,...,x_{n+1}}_\tau + \omega_\tau \para{x_1,...,x_{n+1}} c.
\end{align*}
The map $\omega_\tau \para{x_1,...,x_{n+1}} = \sum_{i=1}^{n+1} (-1)^{i-1} \tau\para{x_i} \omega(x_1,...,\hat{x}_i,...,x_{n+1})$ is a skew-symmetric $(n+1)$-linear form, and $\bracket{\cdot,...,\cdot}_{c,\tau}$ satisfies the Hom-Nambu identity. We have also:
\begin{align*}
\bracket{x_1,...,x_n,c}_{c,\tau}&= \sum_{i=1}^{n} (-1)^{i-1} \tau\para{x_i}\bracket{x_1,...,\hat{x}_i,...,c}_c + (-1)^n \tau(c) \bracket{x_1,...,x_n} \\
&= 0. \qquad \Big(\bracket{x_{i_1},...,x_{i_{n-1}},c}_c = 0 \text{ and } \tau\para{c} = 0.\Big)
\end{align*}
Therefore $\para{\bar{A},\bracket{\cdot,...,\cdot}_{c,\tau},\alpha_c}$ is a central extension of $(A,\bracket{\cdot,...,\cdot}_\tau,\alpha)$. 
\end{proof}

%%%%%%%%%%%%%%%%%%%%%%%%%%%%%%%%%%%
%%%%%%%%%%%%%%%%%%%%%%%%%%%%%%%%%%%

\section{Cohomology of $(n+1)$-Hom-Lie algebras induced by $n$-Hom-Lie algebras} \label{sec:cohomology}

In this section, we study the connections between  the cohomology of a given $n$-Hom-Lie algebra and the cohomology of the induced $(n+1)$-Hom-Lie algebra.

Let $(A,\bracket{\cdot,...,\cdot},\alpha)$ be a multiplicative $n$-Hom-Lie algebra, $\tau$ a trace satisfying $\tau \circ \alpha = \tau$, and $(A,\bracket{\cdot,...,\cdot}_\tau,\alpha)$ the induced algebra. Then we have the following correspondence between 1-cocycles and 2-cocycles of $(A,\bracket{\cdot,...,\cdot},\alpha)$ and those of $(A,\bracket{\cdot,...,\cdot}_\tau,\alpha)$.

\begin{lemma}
If $f : A \to A$ is an $\alpha^k$-derivation of an $n$-Hom-Lie algebra and $\tau$ is a trace map, then $\tau \circ f$ is a trace. 
\end{lemma}
\begin{proof}
For all $x,y \in A$, we have 
\begin{align*}
\tau\para{ f\para{ \bracket{x_1,...,x_n} } } &= \tau\para{ \sum_{i=1}^n \bracket{\alpha^k(x_1),...,\alpha^k(x_{i-1}),f(x_i),\alpha^k(x_{i+1}),...,\alpha^k(x_n) } }\\
&=  \sum_{i=1}^n \tau\para{ \bracket{\alpha^k(x_1),...,\alpha^k(x_{i-1}),f(x_i),\alpha^k(x_{i+1}),...,\alpha^k(x_n) } } = 0.
\end{align*} 
\end{proof}

\begin{proposition}
Let $f:A \to A$ be an $\alpha^k$-derivation of the $n$-Hom-Lie algebra $A$, then $f$ is a derivation of the induced $(n+1)$-Lie algebra if and only if
\[ \alpha^k\para{ \bracket{x_1,...,x_{n+1 }}_{\tau \circ f} } =0, \forall x_1,...,x_{n+1 } \in A. \]

\end{proposition}

\begin{proof}
Let $f$ be a derivation of $A$ and $x_1,...,x_{n+1} \in A$:
\begin{align*}
f&\para{\bracket{x_1,...,x_{n+1}}_\tau} = f\para{\sum_{i=1}^{n+1} (-1)^{i-1} \tau(x_i)\bracket{x_1,...,x_{i-1},x_{i+1},...,x_{n+1} } }\\
&= \sum_{i=1}^{n+1} (-1)^{i-1} \tau(x_i) f\para{\bracket{x_1,...,x_{i-1},x_{i+1},...,x_{n+1} } }\\
&= \sum_{i=1}^{n+1} (-1)^{i-1} \tau(x_i) \sum_{j=1 ; j\neq i}^{n+1} \bracket{\alpha^k(x_1),...,f(x_j),...,\alpha^k(x_{i-1}),\alpha^k(x_{i+1}),...,\alpha^k(x_{n+1}) }\\
&= \sum_{j=1}^{n+1}  \sum_{i=1 ; i\neq j}^{n+1} (-1)^{i-1} \tau(x_i) \bracket{\alpha^k(x_1),...,f(x_j),...,\alpha^k(x_{i-1}),\alpha^k(x_{i+1}),...,\alpha^k(x_{n+1}) } \\
&+ \sum_{j=1}^{n+1} (-1)^{j-1} \tau\para{f(x_j)}\bracket{\alpha^k(x_1),...,\alpha^k(x_{j-1}),\alpha^k(x_{j+1}),...,\alpha^k(x_{n+1})}\\
&- \sum_{j=1}^{n+1} (-1)^{j-1} \tau\para{f(x_j)}\bracket{\alpha^k(x_1),...,\alpha^k(x_{j-1}),\alpha^k(x_{j+1}),...,\alpha^k(x_{n+1})}  \\
&= \sum_{j=1}^{n+1} \sum_{i=1 ; i\neq j}^{n+1} (-1)^{i-1} \tau(\alpha^k(x_i)) \bracket{\alpha^k(x_1),...,f(x_j),...,\alpha^k(x_{i-1}),\alpha^k(x_{i+1}),...,\alpha^k(x_{n+1}) }\\ 
&+ \sum_{j=1}^{n+1} (-1)^{j-1} \tau\para{f(x_j)}\bracket{\alpha^k(x_1),...,\alpha^k(x_{j-1}),\alpha^k(x_{j+1}),...,\alpha^k(x_{n+1})} \\
&- \sum_{j=1}^{n+1} (-1)^{j-1} \tau\para{f(x_j)}\bracket{\alpha^k(x_1),...,\alpha^k(x_{j-1}),\alpha^k(x_{j+1}),...,\alpha^k(x_{n+1})}\\
&= \sum_{j=1}^{n+1} \bracket{\alpha^k(x_1),...,\alpha^k(x_{j-1}),f(x_j),\alpha^k(x_{j+1}),...,\alpha^k(x_{n+1})}_\tau \\ 
& - \alpha^k\para{ \sum_{j=1}^{n+1} (-1)^{j-1} (\tau\circ f)\para{x_j}\bracket{x_1,...,x_{j-1},x_{j+1},...,x_{n+1}} }.
\end{align*}
\end{proof}

Now, we consider the 2-cocycles of an $(n+1)$-Hom-Lie algebra induced by an $n$-Hom-Lie algebra.

\begin{proposition}\label{hom-Z2ad}
Let $(A,\bracket{\cdot,...,\cdot},\alpha)$ be a multiplicative $n$-Hom-Lie algebra, $\tau$ be a trace satisfying conditions of Theorem \ref{thm:inducedmul} and $(A,\bracket{\cdot,...,\cdot}_\tau,\alpha)$ be the induced $(n+1)$-Hom-Lie algebra. Let $\varphi \in Z^2_{ad}(A,A)$ such that
$ \tau\circ \varphi = 0.$

Then $\varphi_\tau: L(A_\tau)\wedge A \to A$ defined by:
\[ \varphi_\tau \para{X,z} = \sum_{i=1}^n (-1)^{i-1} \tau(x_i) \varphi(X_i , z)+(-1)^n \tau(z)\varphi(X_n , x_n),\]
is a 2-cocycle of the induced $(n+1)$-Hom-Lie algebra. Where for $X=x_1\wedge ...\wedge x_n) \in L(A_\tau)$, $X_i=x_1\wedge ...\wedge x_{i-1} \wedge x_{i+1} \wedge ... \wedge x_n$.
\end{proposition}
\begin{proof}
Let $\varphi \in Z^2_{ad}(A,A)$ satisfying the condition above, and let 
\[ \varphi_\tau\para{X,z} = \sum_{i=1}^n (-1)^{i-1} \tau(x_i) \varphi(X_i , z)+(-1)^n \tau(z)\varphi(X_n , x_n).\] Then we have:

\begin{align*}
d^2& \varphi_\tau(X,Y,z) = - \varphi_\tau([X , Y]_\alpha, \alpha (z)) - \varphi_\tau(\bar{\alpha}(Y), X\cdot z) + \varphi_\tau(\bar{\alpha}(X), Y \cdot z) \\
&-\para{\varphi_\tau(X, \quad)\cdot_\alpha Y}\cdot \alpha(z) - \bar{\alpha}(Y)\cdot \varphi_\tau (X,z) + \bar{\alpha}(X) \cdot \varphi_\tau(Y,z)\\
%%%%%%%%%%%%%%%%%%%%%%%%%%%%%%%%%%%%%%%%%%%
&= - \sum_{j=1}^n \sum_{k=1}^n (-1)^{j+k}\tau(x_j)\tau(\alpha(y_k)) \varphi([X_j, Y_k]_\alpha , \alpha(z))\\
&- \sum_{j=1}^n \sum_{i=1}^n (-1)^{j+n-1} \tau(x_j)\tau(\alpha(z)) \varphi(\alpha(y_1),...,\alpha(y_{i-1}), X_j \cdot y_i, \alpha(y_{i+1}),...,\alpha(y_n))\\
&- \sum_{i=1}^n \sum_{k=1 ; k\neq i}^n (-1)^{k+n-1} \tau(y_i)\tau(\alpha(y_k)) \varphi(\alpha(y_1),...,\widehat{y_k},...,\alpha(y_{i-1}), X_n \cdot x_n, \alpha(y_{i+1}),...,\alpha(y_n),\alpha(z))\\
&- \sum_{i=1}^n \tau(y_i)\tau(\alpha(z)) \varphi(\alpha(y_1),...,\alpha(y_{i-1}), X_n \cdot x_n, \alpha(y_{i+1}),...,\alpha(y_n))\\
%%%%%
&- \sum_{j=1}^n \sum_{k=1}^n (-1)^{j+k} \tau(x_j)\tau(\alpha(y_k)) \varphi(\bar{\alpha}(Y_k),X_j \cdot z) - \sum_{i=1}^n (-1)^{i+n-1} \tau(\alpha(y_i))\tau(z) \varphi(\bar{\alpha}(Y_i), X_n \cdot x_n)\\
%%%%%
&+ \sum_{j=1}^n \sum_{k=1}^n (-1)^{j+k} \tau(\alpha(x_j)) \tau(y_k) \varphi(\bar{\alpha}(X_j), Y_k \cdot z) + \sum_{j=1}^n (-1)^{j+n-1} \tau(\alpha(x_j))\tau(z) \varphi(\bar{\alpha}(X_j),Y_n \cdot y_n)\\
%%%%%%
& - \sum_{j=1}^n \sum_{k=1}^n (-1)^{j+k} \tau(x_j) \tau(\alpha(y_k)) \para{\varphi_\tau(X_j, \quad)\cdot_\alpha Y_k}\cdot \alpha(z)\\
&- \sum_{i=1}^n \sum_{j=1}^n (-1)^{n+j-1} \tau(x_j)\tau(\alpha(z)) \bracket{\alpha(y_1), ...,\alpha(y_{i-1}),\varphi(X_j,y_i),...,\alpha(y_n)} \\
&- \sum_{i=1}^n \sum_{k=1 ; k \neq i}^n (-1)^{n+k-1} \tau(y_i)\tau(\alpha(y_k)) \bracket{\alpha(y_1),...,\widehat{y_k},...,\alpha(y_{i-1}),\varphi(X_n,x_n),...,\alpha(y_n),\alpha(z)}\\
&- \sum_{i=1}^n \tau(y_i)\tau(\alpha(z)) \bracket{\alpha(y_1),...,\alpha(y_{i-1}),\varphi(X_n,x_n),...,\alpha(y_n)}\\
%%%%%%
&- \sum_{j=1}^n \sum_{k=1}^n (-1)^{j+k} \tau(x_j)\tau(\alpha(y_k)) \bar{\alpha}(Y_k)\cdot \varphi(X_j, z) 
- \sum_{i=1}^n (-1)^{i+n-1} \tau(\alpha(y_i))\tau(z) \bar{\alpha}(Y_i) \cdot \varphi(X_n, x_n)\\
%%%%%%
&+ \sum_{j=1}^n \sum_{k=1}^n (-1)^{j+k} \tau(\alpha(x_j)) \tau(y_k) \bar{\alpha}(X_j) \cdot \varphi(Y_k , z)+ \sum_{j=1}^n (-1)^{j+n-1} \tau(\alpha(x_j))\tau(z) \bar{\alpha}(X_j)\cdot \varphi(Y_n, y_n) \\
%%%%%%%
&- \sum_{i=1}^n \sum_{j=1}^n (-1)^{i+j} \tau(x_j) \tau(\varphi(X_j,y_i)) \bar{\alpha}(Y_i \cdot z) 
- \sum_{i=1}^n (-1)^{i+n-1} \tau(y_i) \tau(\varphi(X_n,x_n)) \bar{\alpha}(Y_i \cdot z) \\
&- \sum_{i=1}^n (-1)^{n+i-1} \tau(x_i) \tau(\varphi(X_i,z) \bar{\alpha}(Y_n \cdot y_n) -\tau(z)\tau(\varphi(X_n,x_n)) \bar{\alpha}(Y_n \cdot y_n) \\
&+ \sum_{i=1}^n (-1)^{n+i-1} \tau(y_i) \tau(\varphi(Y_i,z) \bar{\alpha}(X_n \cdot x_n) +\tau(z)\tau(\varphi(Y_n,y_n)) \bar{\alpha}(X_n \cdot x_n) \\
%%%%%%%%%%%%%%%%%%%%%%%%%%%%%%%%%%%%%%%%%%%
&= \sum_{j=1}^n \sum_{k=1}^n (-1)^{j+k} \tau(x_j) \tau(x_k) d^2\varphi(X_j,Y_k,z) \\ 
& - \sum_{j=1}^n (-1)^{j+n-1} \tau(x_j)\tau(z) \sum_{i=1}^{n-1} \varphi(\alpha(y_1),...,\alpha(y_{i-1}),X_j \cdot y_i,\alpha(y_{i+1}),...,\alpha(y_{n-1}),\alpha(y_n)) \\
&- \sum_{j=1}^n (-1)^{j+n-1} \tau(x_j)\tau(z) \varphi(\bar{\alpha}(Y_n), X_j \cdot y_n)
+ \sum_{j=1}^n (-1)^{j+n-1} \tau(x_j)\tau(z) \varphi(\bar{\alpha}(X_j), Y_n \cdot y_n) \\
& - \sum_{j=1}^n (-1)^{j+n-1} \tau(x_j)\tau(z) \sum_{i=1}^{n-1} \bracket{\alpha(y_1),...,\alpha(y_{i-1}),X_j \cdot y_i,\alpha(y_{i+1}),...,\alpha(y_{n-1}),\alpha(y_n)} \\
&- \sum_{j=1}^n (-1)^{j+n-1} \tau(x_j)\tau(z) \bar{\alpha}(Y_n) \cdot \varphi( X_j, y_n)
+ \sum_{j=1}^n (-1)^{j+n-1} \tau(x_j)\tau(z) \bar{\alpha}(X_j)\cdot \varphi(Y_n, y_n) \\ 
%%%%%%
&- \sum_{i=1}^n \tau(y_i) \tau(z) \para{(-1)^{n-i}+(-1)^{n+i-1}}\varphi(\bar{\alpha}(Y_i), X_n \cdot x_n)\\
&- \sum_{i=1}^n \tau(y_i)\tau(\alpha(z)) \varphi(\alpha(y_1),...,\alpha(y_{i-1}), X_n \cdot x_n, \alpha(y_{i+1}),...,\alpha(y_n))\\
&- \sum_{i=1}^n \sum_{k=1 ; k\neq i}^n (-1)^{k+n-1} \tau(y_i)\tau(y_k) \varphi(\alpha(y_1),...,\widehat{y_k},...,\alpha(y_{i-1}), X_n \cdot x_n, \alpha(y_{i+1}),...,\alpha(y_n),\alpha(z))\\
&- \sum_{i=1}^n \sum_{k=1 ; k \neq i}^n (-1)^{n+k-1} \tau(y_i)\tau(y_k) \bracket{\alpha(y_1),...,\widehat{y_k},...,\alpha(y_{i-1}),\varphi(X_n,x_n),...,\alpha(y_n),\alpha(z)}\\
&- \sum_{i=1}^n \tau(y_i)\tau(z) \bracket{\alpha(y_1),...,\alpha(y_{i-1}),\varphi(X_n,x_n),...,\alpha(y_n)}\\
& - \sum_{i=1}^n (-1)^{i+n-1} \tau(\alpha(y_i))\tau(z) \bar{\alpha}(Y_i) \cdot \varphi(X_n, x_n) \\
%%%%%%
&- \sum_{i=1}^n \sum_{j=1}^n (-1)^{i+j} \tau(x_j) \tau(\varphi(X_j,y_i)) \bar{\alpha}(Y_i \cdot z) 
- \sum_{i=1}^n (-1)^{i+n-1} \tau(y_i) \tau(\varphi(X_n,x_n)) \bar{\alpha}(Y_i \cdot z) \\
&- \sum_{i=1}^n (-1)^{n+i-1} \tau(x_i) \tau(\varphi(X_i,z) \bar{\alpha}(Y_n \cdot y_n) -\tau(z)\tau(\varphi(X_n,x_n)) \bar{\alpha}(Y_n \cdot y_n) \\
&+ \sum_{i=1}^n (-1)^{n+i-1} \tau(y_i) \tau(\varphi(Y_i,z) \bar{\alpha}(X_n \cdot x_n) +\tau(z)\tau(\varphi(Y_n,y_n)) \bar{\alpha}(X_n \cdot x_n) \\
%%%%%%%%%%%%%%%%%%%%%%%%%%%%%%%%%%%%%%%%%%%
& =  \sum_{j=1}^n \sum_{k=1}^n (-1)^{j+k} \tau(x_j) \tau(x_k) d^2\varphi(X_j,Y_k,z)  + \sum_{j=1}^n (-1)^{j+n-1} \tau(x_j)\tau(z) d^2 \varphi(X_j,Y_n,y_n)\\
%%%%%%
&- \sum_{i=1}^n \sum_{k=1 ; k\neq i}^n (-1)^{k+n-1} \tau(y_i)\tau(y_k) \varphi(\alpha(y_1),...,\widehat{y_k},...,\alpha(y_{i-1}), X_n \cdot x_n, \alpha(y_{i+1}),...,\alpha(y_n),\alpha(z))\\
&- \sum_{i=1}^n \tau(y_i) \tau(z) \para{(-1)^{n-i}+(-1)^{n+i-1}}\varphi(\bar{\alpha}(Y_i), X_n \cdot x_n)\\
&- \sum_{i=1}^n \tau(y_i)\tau(z) \varphi(\alpha(y_1),...,\alpha(y_{i-1}), X_n \cdot x_n, \alpha(y_{i+1}),...,\alpha(y_n))\\
&- \sum_{i=1}^n \sum_{k=1 ; k \neq i}^n (-1)^{n+k-1} \tau(y_i)\tau(y_k) \bracket{\alpha(y_1),...,\widehat{y_k},...,\alpha(y_{i-1}),\varphi(X_n,x_n),...,\alpha(y_n),\alpha(z)}\\
&- \sum_{i=1}^n \tau(y_i)\tau(z) \bracket{\alpha(y_1),...,\alpha(y_{i-1}),\varphi(X_n,x_n),...,\alpha(y_n)}\\
& - \sum_{i=1}^n (-1)^{i+n-1} \tau(y_i)\tau(z) \bar{\alpha}(Y_i) \cdot \varphi(X_n, x_n) \\
%%%%%%%
&- \sum_{i=1}^n \sum_{j=1}^n (-1)^{i+j} \tau(x_j) \tau(\varphi(X_j,y_i)) \bar{\alpha}(Y_i \cdot z) 
- \sum_{i=1}^n (-1)^{i+n-1} \tau(y_i) \tau(\varphi(X_n,x_n)) \bar{\alpha}(Y_i \cdot z)  \\
&- \sum_{i=1}^n (-1)^{n+i-1} \tau(x_i) \tau(\varphi(X_i,z) \bar{\alpha}(Y_n \cdot y_n) -\tau(z)\tau(\varphi(X_n,x_n)) \bar{\alpha}(Y_n \cdot y_n) \\
&+ \sum_{i=1}^n (-1)^{n+i-1} \tau(y_i) \tau(\varphi(Y_i,z) \bar{\alpha}(X_n \cdot x_n) +\tau(z)\tau(\varphi(Y_n,y_n)) \bar{\alpha}(X_n \cdot x_n) \\
&= 0. \qquad \text{by Lemma \ref{lemma:DSG} and the condition } \tau \circ \varphi =0.
\end{align*}

\end{proof}

\begin{proposition}
Every $1$-cocycle for the scalar cohomology of a multiplicative $n$-Hom-Lie algebra $(A,\bracket{\cdot,...,\cdot},\alpha)$ is a $1$-cocycle for the scalar cohomology of the induced algebra. Notice that $1$-cocycles for the scalar cohomology are exactly traces.
\end{proposition}
\begin{proof}
Let $\omega$ be a $1$-cocycle for the scalar cohomology of $(A,\bracket{\cdot,...,\cdot},\alpha)$, then 
\[\forall x_1,...,x_{n-1},z \in A, d^1 \omega(x_1,...,x_{n-1},z) = \omega \para{\bracket{x_1,...,x_{n-1},z}} = 0,\]
 which is equivalent to $\bracket{A,...,A} \subset \ker \omega$. By Remark \ref{Hom-D3subD}, $\bracket{A,...,A}_\tau \subset \bracket{A,...,A}$ and then $\bracket{A,...,A}_\tau \subset \ker \omega$, that is 
\[ \forall x_1,...,x_n,z \in A, \omega\para{\bracket{x_1,...,x_n,z}_\tau}=d^1 \omega \para{x_1,...,x_n,z} = 0.\]
 It means that $\omega$ is a $1$-cocycle for the scalar cohomology of $\para{A,\bracket{\cdot,...,\cdot}_\tau,\alpha}$.  
\end{proof}

\begin{proposition}\label{Hom-Z2tri}
Let $\varphi \in Z^2_{0}(A,\mathbb{K})$, then the map $\varphi_\tau: L(A_\tau)\wedge A \to \K$ defined by:
\[ \varphi_\tau\para{X,z} = \sum_{i=1}^n (-1)^{i-1} \tau(x_i) \varphi(X_i , z)+(-1)^n \tau(z)\varphi(X_n , x_n)\] is a $2$-cocycle of the induced $(n+1)$-Hom-Lie algebra,
where for $X=x_1 \wedge ... \wedge x_n \in L(A_\tau)$, $X_i=x_1 \wedge ... \wedge x_{i-1} \wedge x_{i+1} \wedge ... \wedge x_n)$.
\end{proposition}
\begin{proof}
Let $\varphi \in Z^2_{0}(A,\mathbb{K})$ satisfying the condition above, and let \[\varphi_\tau\para{X,z} = \sum_{i=1}^n (-1)^{i-1} \tau(x_i) \varphi(X_i , z)+(-1)^n \tau(z)\varphi(X_n , x_n).\] Then we have:
\begin{align*}
d^2 &\varphi_\tau(X,Y,z) = - \varphi_\tau([X , Y]_\alpha, \alpha (z)) - \varphi_\tau(\bar{\alpha}(Y), X\cdot z) + \varphi_\tau(\bar{\alpha}(X), Y \cdot z) \\
&= - \sum_{j=1}^n \sum_{k=1}^n (-1)^{j+k}\tau(x_j)\tau(\alpha(y_k)) \varphi([X_j , Y_k]_\alpha ,\alpha(z) )\\
&- \sum_{j=1}^n \sum_{i=1}^n (-1)^{j+n-1} \tau(x_j)\tau(\alpha(z)) \varphi(\alpha(y_1),...,\alpha(y_{i-1}), X_j \cdot y_i, \alpha(y_{i+1}),...,\alpha(y_n))\\
&- \sum_{i=1}^n \sum_{k=1 ; k\neq i}^n (-1)^{k+n-1} \tau(y_i)\tau(\alpha(y_k)) \varphi(\alpha(y_1),...,\widehat{y_k},...,\alpha(y_{i-1}), X_n \cdot x_n, \alpha(y_{i+1}),...,\alpha(y_n),\alpha(z))\\
&- \sum_{i=1}^n \tau(y_i)\tau(\alpha(z)) \varphi(\alpha(y_1),...,\alpha(y_{i-1}), X_n \cdot x_n, \alpha(y_{i+1}),...,\alpha(y_n))\\
&- \sum_{j=1}^n \sum_{k=1}^n (-1)^{j+k} \tau(x_j)\tau(\alpha(y_k)) \varphi(\bar{\alpha}(Y_k),X_j \cdot z) - \sum_{i=1}^n (-1)^{i+n-1} \tau(\alpha(y_i))\tau(z) \varphi(\bar{\alpha}(Y_i), X_n \cdot x_n)\\
&+ \sum_{j=1}^n \sum_{k=1}^n (-1)^{j+k} \tau(\alpha(x_j)) \tau(y_k) \varphi(\bar{\alpha}(X_j), Y_k \cdot z) + \sum_{j=1}^n (-1)^{j+n-1} \tau(\alpha(x_j))\tau(z) \varphi(\bar{\alpha}(X_j),Y_n \cdot y_n)\\
%%%%%%%%%%%%%%%%%%%%%%%%%%%%%%%%%%
&= \sum_{j=1}^n \sum_{k=1}^n (-1)^{j+k} \tau(x_j) \tau(y_k) d^2\varphi(X_j,Y_k,z) \\
%%%%%%%%%%
& - \sum_{j=1}^n (-1)^{j+n-1} \tau(x_j)\tau(\alpha(z)) \sum_{i=1}^{n-1} \varphi(\alpha(y_1),...,\alpha(y_{i-1}),X_j \cdot y_i,\alpha(y_{i+1}),...,\alpha(y_{n-1}),\alpha(y_n))\\
& -\sum_{j=1}^n (-1)^{j+n-1} \tau(x_j)\tau(\alpha(z)) \varphi(\bar{\alpha}(Y_n), X_j \cdot y_n)
+ \sum_{j=1}^n (-1)^{j+n-1} \tau(\alpha(x_j))\tau(z) \varphi(\bar{\alpha}(X_j), Y_n \cdot y_n)\\
%%%%%%%%%%%
&- \sum_{i=1}^n \sum_{k=1 ; k\neq i}^n (-1)^{k+n-1} \tau(y_i)\tau(\alpha(y_k)) \varphi(\alpha(y_1),...,\widehat{y_k},...,\alpha(y_{i-1}), X_n \cdot x_n, \alpha(y_{i+1}),...,\alpha(y_n),\alpha(z))\\
&- \sum_{i=1}^n \tau(y_i)\tau(\alpha(z)) \varphi(\alpha(y_1),...,\alpha(y_{i-1}), X_n \cdot x_n, \alpha(y_{i+1}),...,\alpha(y_n))\\
&- \sum_{i=1}^n (-1)^{i+n-1} \tau(\alpha(y_i))\tau(z) \varphi(\bar{\alpha}(Y_i), X_n \cdot x_n)\\
%%%%%%%%%%%%%%%%%%%%%%%%%%%%%
& =  \sum_{j=1}^n \sum_{k=1}^n (-1)^{j+k} \tau(x_j) \tau(y_k) d^2\varphi(X_j,Y_k,z) + \sum_{j=1}^n (-1)^{j+n-1} \tau(x_j)\tau(z) d^2 \varphi(X_j,Y_n,y_n)\\
%%%%%%%%%%
&- \sum_{i=1}^n \sum_{k=1 ; k\neq i}^n (-1)^{k+n-1} \tau(y_i)\tau(y_k) \varphi(\alpha(y_1),...,\widehat{y_k},...,\alpha(y_{i-1}), X_n \cdot x_n, \alpha(y_{i+1}),...,\alpha(y_n),\alpha(z))\\
&- \sum_{i=1}^n \tau(y_i)\tau(z) \varphi(\alpha(y_1),...,\alpha(y_{i-1}), X_n \cdot x_n, \alpha(y_{i+1}),...,\alpha(y_n))\\
&- \sum_{i=1}^n (-1)^{i+n-1} \tau(y_i)\tau(z) \varphi(\bar{\alpha}(Y_i), X_n \cdot x_n)
= 0. \qquad \text{by Lemma \ref{lemma:DSG}.}
\end{align*}

\end{proof}

\begin{lemma}\label{coBtau}
Let $\omega \in C^1(A,\mathbb{K})$. Then:
\[ d_\tau^1 \omega\para{x_1,...,x_{n+1}} = \sum_{i=1}^{n+1} (-1)^{i-1} \tau(x_i) d^1\omega\para{x_1,...,\widehat{x_i},...,x_{n+1}}, \forall x_1,...,x_{n+1} \in A, \]
where $d^p_\tau$ is the coboundary operator for the cohomology complex of $A_\tau$.
\end{lemma}
\begin{proof}
Let $\omega \in C^1(A,\mathbb{K})$, $x_1,...,x_{n+1} \in A$, then we have:
\begin{align*}
&d_\tau^1 \omega\para{x_1,...,x_{n+1}} = \omega\para{\bracket{x_1,...,x_{n+1}}_\tau}\\
&=\sum_{i=1}^{n+1} (-1)^{i-1} \tau(x_i) \omega\para{\bracket{x_1,...,\widehat{x_i},...,x_{n+1}}}
=  \sum_{i=1}^{n+1} (-1)^{i-1} \tau(x_i) d^1\omega\para{x_1,...,\widehat{x_i},...,x_{n+1}}.
\end{align*}

\end{proof}

\begin{proposition}\label{B2triv_eq}
Let $\varphi_1,\varphi_2 \in Z^2_{0}(A,\mathbb{K})$. If $\varphi_1,\varphi_2$ are in the same cohomology class then $\psi_1,\psi_2$ defined by:
\[ \psi_i\para{x_1,...,x_{n+1}} = \sum_{j=1}^{n+1} (-1)^{j-1} \tau \para{x_j} \varphi_i \para{x_1,...,\widehat{x_j},...,x_{n+1}}, i=1,2, \]
are in the same cohomology class.
\end{proposition}
\begin{proof}
Let $\varphi_1,\varphi_2 \in Z^2_{0}(A,\mathbb{K})$ be two cocycles in the same cohomology class, that is \[\varphi_2 - \varphi_1 =d^1\alpha,\  \alpha \in C^1(A,\mathbb{K}),\] and 
\[ \psi_i\para{x_1,...,x_{n+1}} = \sum_{j=1}^{n+1} (-1)^{j-1} \tau \para{x_j} \varphi_i \para{x_1,...,\widehat{x_j},...,x_{n+1}}, i=1,2. \]
Then we have:
\begin{align*}
&\psi_2\para{x_1,...,x_{n+1}}-\psi_1\para{x_1,...,x_{n+1}} = \sum_{j=1}^{n+1} (-1)^{j-1} \tau \para{x_j} (\varphi_2 \para{x_1,...,\widehat{x_j},...,x_{n+1}} \\
&- \sum_{j=1}^{n+1} (-1)^{j-1} \tau \para{x_j}  \varphi_1 \para{x_1,...,\widehat{x_j},...,x_{n+1}})\\
&= \sum_{j=1}^{n+1} (-1)^{j-1} \tau \para{x_j} (\varphi_2 -\varphi_1) \para{x_1,...,\widehat{x_j},...,x_{n+1}}\\
&= \sum_{j=1}^{n+1} (-1)^{j-1} \tau \para{x_j} d^1\alpha \para{x_1,...,\widehat{x_j},...,x_{n+1}}
= d_\tau^1 \alpha\para{x_1,...,x_{n+1}}.
\end{align*}
That means that $\psi_1$ and $\psi_2$ are in the same cohomology class. 
\end{proof}

%%%%%%%%%%%%%%%%%%%%%%%%%%%%%%%%%%%
%%%%%%%%%%%%%%%%%%%%%%%%%%%%%%%%%%%
%%%%%%%%%%%%%%%%%%%%%%%%%%%%%%%%%%%
%%%%%%%%%%%%%%%%%%%%%%%%%%%%%%%%%%%

\section{Examples} \label{sec:examples}
%Introduce
\begin{example}
We consider the $4$-dimensional Lie algebra $(A,\bracket{\cdot,\cdot})$ defined, in the basis $(e_i)_{1\leq i \leq 4}$, by 
$$ \bracket{e_1,e_2}=e_2 ; \bracket{e_3,e_4}=e_4. $$
A linear map $\alpha : A \to A$ having, in the basis $(e_i)_{1\leq i \leq 4}$, the matrix $T=(a_{ij})_{1\leq i,j \leq 4}$ is a Lie algebra endomorphism if and only if
$ \alpha\para{\bracket{e_i,e_j}}=\bracket{\alpha(e_i),\alpha(e_j)}, \quad 1\leq i < j \leq 4. $\\
From these conditions, we get the following equations:
\[%%%%%%%%%
a_{22}(a_{11}-1)=0 \ ;\ 
a_{42}(a_{31}-1)=0 \ ;\ 
a_{24}(a_{13}-1)=0 \ ;\ 
a_{44}(a_{33}-1)=0 ;\]
\[
a_{12}=a_{32}=a_{14}=a_{34}=0 \ ;\ 
a_{11} a_{23} - a_{13} a_{21}=0 \ ;\ 
a_{31} a_{43} - a_{33} a_{41}=0 ;\]
\[
a_{22} a_{13} = 0 \ ;\ 
a_{42} a_{33} = 0 \ ;\ 
a_{11} a_{24} = 0 \ ;\ 
a_{31} a_{44} = 0.\]
%%%%%%%%%%
We choose the endomorphism $\alpha_1$ given, in the basis $(e_i)_{1\leq i \leq 4}$, by the matrix:
\[ T_1= 
\begin{pmatrix}
0 & 0 & 1 & 0\\
0 & 0 & c & d\\
1 & 0 & 0 & 0\\
a & b & 0 & 0
\end{pmatrix} \qquad b,d \neq 0 \]
  
The bracket $\bracket{\cdot,\cdot}_1=\alpha_1 \circ \bracket{\cdot,\cdot}$ is given by 
$$ \bracket{e_1,e_2}_1 = b e_4 ; \bracket{e_3,e_4}_1 = d e_2. $$
The algebra $(A,\bracket{\cdot,\cdot}_1,\alpha_1)$ is a (multiplicative) Hom-Lie algebra (Theorem \ref{twist}). We further refer to this algebra by $A_1$.

We consider now the algebra $(A,\bracket{\cdot,\cdot}_1,\alpha_1)$. We have that $\bracket{A,A}_1 = \left\langle \{ e_2,e_4 \}\right \rangle$. Notice that the bracket $\bracket{\cdot,\cdot}_1$ does not satisfy the Jacobi identity, we have:

\begin{align*}
 \bracket{e_1,\bracket{e_2,e_3}_1}_1 - \bracket{\bracket{e_1,e_2}_1,e_3}_1 -\bracket{e_2,\bracket{e_1,e_3}_1}_1 &= \bracket{e_1, 0}_1 - \bracket{b e_4,e_3}_1 -\bracket{e_2, 0}_1 \\ 
 &= -b \bracket{e_4,e_3}_1= d b e_2 \neq 0.
\end{align*}

Let $\tau : A \to \K$ be a linear map, with 
\[ \tau(x)=\tau\para{\sum_{i=1}^4 x_i e_i} = \sum_{i=1}^4 t_i x_i. \]
The map $\tau$ is a trace if and only if 
\[ t_2=t_4=0. \]
Moreover, a trace map satisfies the condition of Theorem \ref{thm:inducedmul} if and only if 
\[ t_1=t_3. \]

Let $\tau_1 : A \to \K$ be the linear form defined by:
\[ \tau_1(x)=\tau_1\para{\sum_{i=1}^4 x_i e_i} = x_1+x_3. \]
Then by Theorem \ref{thm:inducedmul}, we can construct the induced algebra $(A, \bracket{\cdot,\cdot,\cdot}_{1,\tau_1},\alpha)$, we further refer to it by $A_{\tau_1}$. It is a multiplicative $3$-Hom-Lie algebra. The bracket $\bracket{\cdot,\cdot,\cdot}_{1,\tau_1}$ is given by:
\[\bracket{e_1,e_2,e_3}_{1,\tau_1} = b e_4 ; \bracket{e_1,e_3,e_4}_{1,\tau_1} = d e_2. \]

%Central descending series

We look now at the central descending series of these algebras, we get that
\[C^n(A_1) = \left\langle \{ e_2 , e_4 \}\right\rangle, \]
and
\[C^n(A_{\tau}) = \left\langle \{ e_2 , e_4 \}\right\rangle. \]

%Cohomology

Now, let $\varphi$ be a skew-symmetric bilinear form on $A$. The map $\varphi$ is fully defined by the scalars 
\[\varphi_{ij}=\varphi\para{e_i,e_j}, \quad 1\leq i<j \leq 4.\]
By solving the equations for $\varphi$ to be a $2$-cocycle, that is 
$d ^2 \varphi (e_i,e_j,e_k)=0$ for $1\leq i<j<k\leq 4$, 
we get
$ \varphi_{14}=0 \ ; \ \varphi_{23}=0 \ ;\ \varphi_{24}=0. $\\
Let $\omega$ be a linear form on $A$, defined by
$ \omega(e_i) = \omega_i, 1\leq i \leq 4. $\\
We find that
\[ d^1\omega (e_1,e_2)= d \omega_4 \ ; \  d^1\omega (e_3,e_4)=b \omega_2 \ ; \  d^1\omega (e_i,e_j)=0 \text{ for } (i,j) \neq (1,2),(3,4) \ ; \ i<j.  \]
We get that the second cohomology group for this algebra is $1$-dimensional. Now for a $2$-cocycle $\varphi$, let us consider $\varphi_{\tau_1}$ defined as in Proposition \ref{Hom-Z2tri}, we get:
\[ \varphi_{\tau_1}(e_1,e_2,e_3) = \varphi_{12} \ ;\ \varphi_{\tau_1}(e_1,e_2,e_4) = 0 \ ;\ \varphi_{\tau_1}(e_1,e_3,e_4) = \varphi_{34} \ ;\ \varphi_{\tau_1}(e_2,e_3,e_4) = 0. \]

On the other hand we have, the $2$-cocycles and $2$-coboundaries of the induced algebra $A_{\tau_1}$. For a linear form $\omega$ on $A$, we have 
\[ d_{\tau_1}^1\omega (e_1,e_2,e_3)= d \omega_4 \ ; \  d_{\tau_1}^1\omega (e_1,e_3,e_4)=b \omega_2 \ ; \  d_{\tau_1}^1\omega (e_1,e_2,e_4)=0 ; d_{\tau_1}^1(e_2,e_3,e_4)=0. \]
For $2$-cocycles, we have that every skew-symmetric trilinear form on $A$ is a $2$-cocycle of the induced algebra. Therefore there exist $2$-cocycles of $A_{\tau_1}$ which are not induced by $2$-cocycles of $A_1$.
 
\end{example}

\begin{example}
We take now the same Lie algebra $(A,\bracket{\cdot,\cdot})$ as in the preceding example, and we choose the morphism $\alpha_2$ given by the matrix
\[T_2 = \begin{pmatrix}
a & 0 & e & 0\\
b & 0 & \frac{b e}{a} & 0\\
1 & 0 & 0 & 0\\
c & d & 0 & 0
\end{pmatrix} \qquad a,d \neq 0 ; a \neq 1. \]
We get the bracket $\bracket{\cdot,\cdot}_2=\alpha_2 \circ \bracket{\cdot,\cdot}$ defined by
$ \bracket{e_1,e_2}_2=d e_4. $\\ We further refer to the algebra $(A,\bracket{\cdot,\cdot}_2,\alpha_2)$ by $A_1$.
Let $\tau : A \to \K$ be a linear map, with 
\[ \tau(x)=\tau\para{\sum_{i=1}^4 x_i e_i} = \sum_{i=1}^4 t_i x_i. \]
The map $\tau$ is a trace if and only if 
$ t_4=0. $
Looking for trace maps satisfying the condition of Theorem \ref{thm:inducedmul}, we find that such a trace map exists if and only if $\alpha_2$ satisfies $e=1-a$. We also have that such a trace map satisfies
$ t_2=0 \ ; \ t_3= (1-a) t_1. $\\
Let $\tau_2 : A \to \K$ be the linear form defined by:
\[ \tau_2(x)=\tau_2\para{\sum_{i=1}^4 x_i e_i} = x_1+(1-a) x_3. \]
Then, by Theorem \ref{thm:inducedmul}, we construct the induced multiplicative $3$-Hom-Lie algebra $(A, \bracket{\cdot,\cdot,\cdot}_{2,\tau_2},\alpha)$, we further refer to it by $A_{\tau_2}$. Its bracket is given by:
\[\bracket{e_1,e_2,e_3}_{2,\tau_2} = d e_4. \]

Now, let $\varphi$ be a skew-symmetric bilinear form on $A$. By solving the equations for $\varphi$ to be a $2$-cocycle, that is 
$d ^2 \varphi (e_i,e_j,e_k)=0$ for $1\leq i<j<k\leq 4$, 
we get
$ \varphi_{14}= - \frac{b}{a} \varphi_{24}. $
Coboundaries for this algebra are given by:
\[ d^1\omega (e_1,e_2)= d \omega_4 \ ; \  d^1\omega (e_i,e_j)=0 \text{ for } (i,j) \neq (1,2) \ ; \ i<j,  \]
where $\omega$ is a linear form on $A$. Which means that the second (scalar) cohomology group of this algebra has dimension $4$.
Now for a $2$-cocycle $\varphi$, let us consider $\varphi_{\tau_1}$ defined as in Proposition \ref{Hom-Z2tri}, we get:
\[ \varphi_{\tau_2}(e_1,e_2,e_3) = (1-a) \varphi_{12} + \varphi_{23} \ ;\ \varphi_{\tau_2}(e_1,e_2,e_4) = \varphi_{24} ; \]
 \[ \varphi_{\tau_2}(e_1,e_3,e_4) = \varphi_{34}+\frac{b(1-a)}{a} \varphi_{24} \ ;\ \varphi_{\tau_2}(e_2,e_3,e_4) = (a-1)\varphi_{24}. \]

On the other hand we have, the $2$-cocycles and $2$-coboundaries of the induced algebra $A_{\tau_2}$. For a linear form $\omega$ on $A$, we have 
\[ d_{\tau_2}^1\omega (e_1,e_2,e_3)= d \omega_4 \ ; \  d_{\tau_2}^1\omega (e_1,e_3,e_4)=0  \ ; \  d_{\tau_2}^1\omega (e_1,e_2,e_4)=0 ; d_{\tau_2}^1(e_2,e_3,e_4)=0. \]
For $2$-cocycles, we have that every skew-symmetric trilinear form on $A$ is a $2$-cocycle of the induced algebra. Therefore there exist $2$-cocycles of $A_{\tau_2}$ which are not induced by $2$-cocycles of $A_2$, more precisely, in this case, we have:
\[\{\varphi_{\tau_2} : \varphi \in Z^2(A_2,\K)\} = \{ \psi \in Z^2(A_{\tau_2},\K) : \psi(e_2,e_3,e_4) = (a-1)\psi(e_1,e_2,e_4) \}.\]
%the set of $\varphi_{\tau_2}$ where $\varphi$ is a $2$-cocycle is equal to the set of $2$-cocycles of $A_{\tau_2}$, $\psi$ satisfying $\psi(e_2,e_3,e_4) = (a-1)\psi(e_1,e_2,e_4)$.
\end{example}
%%%%%%%%%%%%%%%%%%%%%%%%%%%%%%%%%%%
%%%%%%%%%%%%%%%%%%%%%%%%%%%%%%%%%%%

\begin{small}

\section*{Acknowledgement}
A. Kitouni is grateful to Division of Applied Mathematics, the research environment Mathematics and Applied Mathematics (MAM) at the School of Education, Culture and Communication at
 M{\"a}lardalen University, 
V{\"a}ster{\aa}s,  Sweden providing support and excellent research environment during his half a year visit to    M{\"a}lardalen   University when part of the work on this paper has been performed.   

\end{small}

 \textbf{A. Kitouni  and A. Makhlouf, }Université de Haute-Alsace, 4 rue des Frères Lumière, 68093 Mulhouse, France,\\ email: abdennour.kitouni@uha.fr, abdenacer.makhlouf@uha.fr\\ 
\textbf{ S. Silvestrov,}  Mälardalens högskola, Box 883, 721 23 Västerås, Sweden, \\ email: sergei.silvestrov@mdh.se

\end{document}